\documentclass[twoside]{article}
\usepackage[cp1251]{inputenc}
\usepackage[T2A]{fontenc}
\usepackage{array}
\usepackage{amssymb}
\usepackage{amsmath}
\usepackage{amsthm}
\usepackage{latexsym}
\usepackage{bm}
\usepackage{enumerate}
\usepackage[dvips]{graphicx}
\usepackage{epsf}
\usepackage{wrapfig}
\usepackage{euscript}
\usepackage[english,russian]{babel}
\usepackage{textcomp}
\newtheorem{theorem}{Theorem}
\newtheorem{lemma}{Lemma}
\newtheorem{definition}{Definition}

\newtheorem{corollary}{Corollary}
\begin{document}
{\selectlanguage{english}
\binoppenalty = 10000 %
\relpenalty   = 10000 %

\pagestyle{headings} \makeatletter
\renewcommand{\@evenhead}{\raisebox{0pt}[\headheight][0pt]{\vbox{\hbox to\textwidth{\thepage\hfill \strut {\small Grigory. K. Olkhovikov}}\hrule}}}
\renewcommand{\@oddhead}{\raisebox{0pt}[\headheight][0pt]{\vbox{\hbox to\textwidth{{Expressive power of basic modal intuitionistic logic}\hfill \strut\thepage}\hrule}}}
\makeatother

\title{Expressive power of basic modal intuitionistic logic\\
as a fragment of classical FOL}
\author{Grigory K. Olkhovikov\\ Dept of Philosophy II, Ruhr-Universit\"{a}t Bochum\\Dept. of Philosophy, Ural Federal University
\\
tel.: +49-15123307070\\
email: grigory.olkhovikov@rub.de, grigory.olkhovikov@gmail.com}
\date{}
\maketitle
\begin{quote}
{\bf Abstract.} The paper treats $4$ different fragments of
first-order logic induced by their respective versions of Kripke
style semantics for modal intuitionistic logic. In order to
capture these fragments, the notion of asimulation, defined in
\cite{Ol13}, is modified and extended to yield Van Benthem type of
semantic characterization of their respective expressive powers.
It is shown further, that this characterization can be easily
carried over to arbitrary first-order definable subclasses of
classical first-order models.
\end{quote}

\begin{quote}
\end{quote}

It was shown in \cite{Ol13}  and \cite{Ol14} that both
intuitionistic first-order logic and its propositional fragment,
viewed as a different fragments of classical first-order logic,
admit of a full analogue of modal characterization
theorem\footnote{See, e.g. \cite[Ch.1, Th. 13]{Blackburn2006}.}
where invariance with respect to bisimulations is replaced with
invariance with respect to (first-order) asimulations
respectively. The present paper extends these results onto the
main versions of basic modal intuitionistic logic.

The layout of the paper is as follows. Section \ref{S:Prel} starts
with notational conventions, after which we introduce the main
variants of Kripke style semantics for the basic modal
intuitionistic system. All in all we consider $4$ different
variants of semantics, of which $2$ are easily discharged by the
versions of clauses employed in \cite{Ol13} and \cite{Ol14} for
basic intutionistic logic. However, for the other $2$ systems
their semantical characterization is less obvious, mainly due to
their treatment of diamond modality.

Section \ref{S:def} then starts with the main task of the present
paper. It contains definitions for all the four variants of
extension of basic asimulation notion to be employed in semantic
characterization of their respective variants of basic modal
intuitionistic logic. We also formulate here the main results of
the paper, although their proofs are postponed till Sections
\ref{S:proof} and \ref{S:other}. Then Section \ref{S:Rest} is
devoted to characterization of modal intuitionistic fragments of
FOL modulo first-order definable classes of models. Section
\ref{S:conclusion} gives conclusions and drafts directions for
future work.

\section{Preliminaries}\label{S:Prel}

\subsection{Notation}\label{S:Not}

A \emph{formula} is a formula of classical predicate logic without
identity whose predicate letters are in vocabulary $\Sigma =
\{\,R^2, R^2_\Box, R^2_\Diamond, P_1^1,\ldots P_n^1,\ldots\,\}$.
We assume $\{\, \bot, \to, \vee, \wedge, \forall, \exists\,\}$ as
the set of basic connectives and quantifiers for this variant of
classical first-order language, which we call \emph{correspondence
language}. A model is a classical first-order model of
correspondence language. We refer to correspondence formulas with
lower-case Greek letters $\varphi$, $\psi$, and $\chi$, and to
sets of correspondence formulas with upper-case Greek letters
$\Gamma$ and $\Delta$. If $\varphi$ is a correspondence formula,
then we associate with it the following finite vocabulary
$\Sigma_\varphi \subseteq \Sigma$ such that $\Sigma_\varphi =
\{\,R^2,R^2_\Box, R^2_\Diamond\,\} \cup \{\,P_i \mid P_i \text{
occurs in }\varphi\,\}$. More generally, we refer with $\Theta$ to
an arbitrary subset of $\Sigma$ such that $R \in \Theta$. If
$\psi$ is a formula and every predicate letter occurring in $\psi$
is in $\Theta$, then we call $\psi$ a $\Theta$-formula.

 We refer to sequence $x_1,\dots, x_n$ of any objects as
$\bar{x}_n$. We identify a sequence consisting of a single element
with this element. If all free variables of a formula $\varphi$
(formulas in $\Gamma$) coincide with a variable $x$, we write
$\varphi(x)$ ($\Gamma(x)$).

By degree of a classical first-order formula we mean the greatest
number of nested quantifiers occurring in it. Degree of a formula
$\varphi$ is denoted by $r(\varphi)$. Its formal definition by
induction on the complexity of $\varphi$ goes as follows:
\begin{align*}
&r(\bot) = r(\varphi) = 0 &&\text{for atomic $\varphi$}\\
&r(\varphi \circ \psi) = max(r(\varphi), r(\psi)) &&\text{for $\circ \in \{\,\wedge, \vee, \to\,\}$}\\
&r(Qx\varphi) = r(\varphi) + 1 &&\text{for $Q \in \{\,\forall,
\exists\,\}$}
\end{align*}

For $k \in \mathbb{N}$, we say that $\Theta$-formula $\varphi(x)$
such that $r(\varphi) \leq k$ is a $(\Theta, x, k)$-formula.

For a binary relation $S$ and any objects $s, t$ we abbreviate the
fact that $s\mathrel{S}t \wedge t\mathrel{S}s$ by
$s\mathrel{\overset{\leftrightarrow}{S}}t$.

We use the following notation for models of classical predicate
logic:
\[
M = \langle U, \iota\rangle, M_1 = \langle U_1, \iota_1\rangle,
M_2 = \langle U_2, \iota_2\rangle,\ldots , M' = \langle U',
\iota'\rangle, M'' = \langle U'', \iota''\rangle,\ldots,
\]
where the first element of a model is its domain and the second
element is its interpretation of predicate letters. If $k \in
\mathbb{N}$ then we write $R_k$ ($R_{\Box k}$, $R_{\Diamond k}$)
as an abbreviation for $\iota_k(R)$ ($\iota_k(R_\Box)$ ,
$\iota_k(R_\Diamond)$). If $a \in U$ then we say that $(M, a)$ is
a pointed model. Further, we say that $\varphi(x)$ is true at $(M,
a)$ and write $M, a \models \varphi(x)$ iff for any variable
assignment $\alpha$ in $M$ such that $\alpha(x) = a$ we have $M,
\alpha \models \varphi(x)$. It follows from this convention that
the truth of a formula $\varphi(x)$ at a pointed model is to some
extent independent from the choice of its only free variable.
Moreover, for $k \in \mathbb{N}$ we will sometimes write $a
\models_k \varphi(x)$ instead of $M_k, a \models \varphi(x)$.

A modal intuitionistic formula is a formula of modal
intuitionistic propositional logic, where $\{\,\bot, \to, \vee,
\wedge, \Box, \Diamond\,\}$ is the set of basic connectives and
modal operators, and $\{\,p_n\mid n \in \mathbb{N}\,\}$ is the set
of propositional letters. We refer to intuitionistic formulas with
letters $I, J, K$, possibly with primes or subscripts.

\subsection{Definitions of basic modal intuitionistic
logic}\label{S:int}

There exist different versions of basic system of modal
intuitionistic logic. In these paper we only consider versions
that have a Kripke-style semantics associated with them, and we
will view these versions via the lens of their respective
Kripke-style semantics.

Quite naturally, Kripke-style semantics for a given version of
intuitionistic modal logic is normally built as an extension of
Kripke semantics for basic intuitionistic propositional logic;
that is to say, the models extend Kripke models for basic
propositional logic and the satisfaction clauses for $\bot$,
$\to$, $\vee$, $\wedge$ are left unchanged (cf. \cite[Definition
7.1 and Definition 7.2]{Mints00})

The new components in the models are one or more additional binary
relations between states which are needed to handle the
satisfaction clauses for $\Box$ and $\Diamond$. In the most
general case both modal operators are handled by separate
relations $R_\Box$ and $R_\Diamond$, although not infrequently one
assumes that $R_\Box$ and $R_\Diamond$ do coincide\footnote{This
is the case, e.g., for all the systems mentioned in
\cite[Ch.3]{Simp1994}.} or are otherwise non-trivially related. It
is also quite common to assume different conditions connecting
$R_\Box$ and $R_\Diamond$ with accessibility relation $R$, like,
e.\,g. assuming that
$$R\circ R_\Box \subseteq R_\Box\circ R.$$

Both the condition that $R_\Box = R_\Diamond$ and the other
conditions mentioned in this connection in the existing literature
are easily first-order definable. Since it was shown in
\cite{Ol13} and \cite{Ol14} that asimulations are easily scalable
according to arbitrary first-order conditions imposed upon the
models, we will first concentrate on the minimal case without any
restrictions imposed and then accommodate for the possible
restrictions in a trivial way, by restricting the domain and the
counter-domain of asimulation relations accordingly.

As for the satisfaction clauses, employed in the existing
literature on Kripke-style semantics for intuitionistic modal
operators, the following variants of them seem to be the most
common and general:\footnote{In this form they are given, e.g. in
\cite[Section 4]{AlShk06}.}
\begin{align}
&M, s \models \Box I \Leftrightarrow \forall t(sR_\Box t
\Rightarrow
M, t \models I)\label{E:box1}\tag{\text{$\Box_1$}}\\
&M, s \models \Box I \Leftrightarrow \forall t(sRt \Rightarrow
\forall u(tR_\Box u \Rightarrow
M, u \models I))\label{E:box2}\tag{\text{$\Box_2$}}\\
&M, s \models \Diamond I \Leftrightarrow \exists t(sR_\Diamond t
\wedge
M, t \models I)\label{E:diam1}\tag{\text{$\Diamond_1$}}\\
&M, s \models \Diamond I \Leftrightarrow \forall t(sRt \Rightarrow
\forall u(tR_\Diamond u \wedge M, u \models
I))\label{E:diam2}\tag{\text{$\Diamond_2$}}
\end{align}
This gives us $4$ possible choices of satisfaction clauses. In
literature, these sets of satisfaction clauses are often viewed as
more or less explicit manifestations of one and the same set of
semantic intuitions. In this view, the reason why clauses
\eqref{E:box2} and \eqref{E:diam2} differ from \eqref{E:box1} and
\eqref{E:diam1}, respectively, is that the former clauses lift up
to the level of semantical definitions some desirable properties
that under \eqref{E:box1} and \eqref{E:diam1} are handled by
restrictions on the class of models.\footnote{E.g. the property of
monotonicity. Cf. the motivation for the clause \eqref{E:box2}
given in \cite[p.46]{Simp1994}.} However, in what follows, we will
disregard this circumstance and will simply consider these four
systems of clauses as \emph{bona fide} different systems. The
reason for this is that, like we said above, we find it convenient
in the context of treating modal intuitionistic formulas via
asimulations, to omit whatever restrictions on models that are
employed to equate these systems in the existing literature on the
subject.

Every of the $4$ semantical choices sketched above induces a
different standard translation of modal intuitionistic formulas
into classical FOL thus giving a different fragment of
it.\footnote{Every such standard translation is an obvious
extension of the well known notion of standard translation of
propositional intuitionistic formulas, see e.g. \cite[Definition
8.7]{Mints00}.} More precisely, for $i,j \in \{ 1,2 \}$ we will
denote the $(i,j)$-standard translation, or the standard
translation induced by adopting $(\Box_i)$-clause together with
$(\Diamond_j)$-clause above, by $ST_{ij}$.

Thus the inductive definitions of the $(i,j)$-standard
$x$-translations run as follows:
\begin{align*}
&ST_{ij}(p_n, x) = P_n(x);\\
&ST_{ij}(\bot, x) = \bot;\\
&ST_{ij}(I \wedge J, x) = ST_{ij}(I, x) \wedge
ST_{ij}(J, x);\\
&ST_{ij}(I \vee J, x) = ST_{ij}(I, x) \vee
ST_{ij}(J, x);\\
&ST_{ij}(I \to J, x) = \forall y(R(x, y) \to (ST_{ij}(I, y) \to
ST_{ij}(J,
y)));\\
&ST_{1j}(\Box I, x) = \forall y(R_\Box(x, y) \to ST_{1j}(I, y));\\
&ST_{2j}(\Box I, x) = \forall y(R(x, y) \to \forall z(R_\Box(y, z) \to ST_{2j}(I, z)));\\
&ST_{i1}(\Diamond I, x) = \exists y(R_\Diamond(x, y) \wedge
ST_{i1}(I, y));\\
&ST_{i2}(\Diamond I, x) = \forall y(R(x, y) \to \exists
z(R_\Diamond(y, z) \wedge ST_{i2}(I, z))).
\end{align*}

Standard conditions are imposed on the variables $x$, $y$, and
$z$.

\section{Characterization of modal intuitionistic formulas: definitions and main
results}\label{S:def}

Our aim in the present paper is to characterize the expressive
power of the four fragments of correspondence language induced by
the four above-mentioned versions of $(i,j)$-standard translation
of modal intuitionistic formulas via the suitable extension of the
notion of asimulation for the intuitionistic propositional logic.
We will begin by giving strict definitions of the four required
extensions, and then formulate the two versions of our main result
for all the four considered fragments of correspondence language
in one full sweep.

It is easy to see that if one chooses \eqref{E:diam1} as a
definition for the semantics of possibility, then the respective
standard translation of modal propositional formulas looks almost
as a notational variant of intuitionistic first-order logic. The
semantics of diamond then resembles the semantics of
intuitionistic existential quantifier, and the semantics of box,
if interpreted according to \eqref{E:box2}, resembles the
semantics of intuitionistic universal quantifier. The only
difference is that the binary relation of existence of object in a
state is replaced by binary relations $R_\Box$ and $R_\Diamond$
respectively. Also, note that now these `quantifiers' are based on
different binary relations rather than one and the same, but the
proofs given in the following section show that this little
wrinkle is of no consequence:

\begin{definition}\label{D:k-asim21}
{\em Let $(M_1, t)$, $(M_2, u)$ be two pointed $\Theta$-models. A
binary relation $A$ is called \emph{$(2,1)$-modal $\langle
(M_1,t), (M_2,u)\rangle_k$-asimulation} iff for any $i,j \in \{ 1,
2 \}$, any $\bar{a}_m, a,c \in U_i$, $\bar{b}_m, b, d \in U_j$,
any unary predicate letter $P \in \Theta$, the following
conditions hold:
\begin{align}
&A \subseteq \bigcup_{n > 0}((U_1^n \times U_2^n) \cup (U_2^n
\times U_1^n))\label{E:c1}\tag{\text{p-type}}\\
&t\mathrel{A}u\label{E:c2}\tag{\text{elem}}\\
&((\bar{a}_m, a)\mathrel{A}(\bar{b}_m, b) \wedge a \models_i P(x)) \Rightarrow b \models_j P(x)\label{E:c3}\tag{\text{p-base}}\\
&((\bar{a}_m, a)\mathrel{A}(\bar{b}_m, b) \wedge b\mathrel{R_j}d
\wedge m < k) \Rightarrow\notag\\
&\qquad\qquad\Rightarrow \exists c \in U_i(a\mathrel{R_i}c \wedge
(\bar{a}_m, a, c)\mathrel{\overset{\leftrightarrow}{A}}(\bar{b}_m,
b, d))\label{E:c4}\tag{\text{p-step}}\\
&((\bar{a}_m, a)\mathrel{A}(\bar{b}_m, b) \wedge b\mathrel{R_j}d
\wedge d\mathrel{R_{\Box j}}f
\wedge m + 1 < k) \Rightarrow\notag\\
&\qquad\qquad\Rightarrow \exists c,e \in U_i(a\mathrel{R_i}c
\wedge c\mathrel{R_{\Box i}}e \wedge (\bar{a}_m, a,
c,e)\mathrel{A}(\bar{b}_m,
b, d,f))\label{E:c5}\tag{\text{p-box-2}}\\
&((\bar{a}_m, a)\mathrel{A}(\bar{b}_m, b) \wedge
a\mathrel{R_{\Diamond i}}c
\wedge m < k) \Rightarrow\notag\\
&\qquad\qquad\Rightarrow \exists d \in U_j(b\mathrel{R_{\Diamond
j}}d \wedge (\bar{a}_m, a, c)\mathrel{A}(\bar{b}_m,
b,d))\label{E:c6}\tag{\text{p-diam-1}}
\end{align}
}
\end{definition}
\begin{definition}\label{D:asim21}
{\em Let $(M_1, t)$, $(M_2, u)$ be two pointed $\Theta$-models. A
binary relation $A$ is called \emph{$(2,1)$-modal $\langle
(M_1,t), (M_2,u)\rangle$-asimulation} iff for any $i,j \in \{ 1, 2
\}$, any $a,c \in U_i$, $b, d \in U_j$, any unary predicate letter
$P \in \Theta$ the following conditions hold:
\begin{align}
&A \subseteq (U_1 \times U_2) \cup (U_2\times
U_1)\label{E:c22}\tag{\text{type}}\\
&t\mathrel{A}u\label{E:c11}\tag{\text{elem}}\\
&(a\mathrel{A}b \wedge a \models_i P(x)) \Rightarrow b \models_j P(x))\label{E:c33}\tag{\text{base}}\\
&(a\mathrel{A}b \wedge b\mathrel{R_j}d) \Rightarrow \exists c \in
U_i(a\mathrel{R_i}c \wedge
c\mathrel{\overset{\leftrightarrow}{A}}d)\label{E:c44}\tag{\text{step}}\\
&(a\mathrel{A}b \wedge b\mathrel{R_j}d \wedge d\mathrel{R_{\Box
j}}f) \Rightarrow \exists
c,e \in U_i(a\mathrel{R_i}c \wedge c\mathrel{R_{\Box i}}e \wedge e\mathrel{A}f)\label{E:c55}\tag{\text{box-2}}\\
&(a\mathrel{A}b \wedge a\mathrel{R_{\Diamond i}}c) \Rightarrow
\exists d \in U_j(b\mathrel{R_{\Diamond j}}d \wedge
c\mathrel{A}d)\label{E:c66}\tag{\text{diam-1}}
\end{align}
}
\end{definition}

The situation changes very little when one adapts \eqref{E:box1}
instead of \eqref{E:box2} as the clause defining box. In fact,
this clause is the standard one from classical modal logic, and to
accommodate for this change one only needs to pick a suitable
asymmentric variant of the bisimulation clause. The resulting
definitions then look as follows:

\begin{definition}\label{D:k-asim11}
{\em Let $(M_1, t)$, $(M_2, u)$ be two pointed $\Theta$-models. A
binary relation $A$ is called \emph{$(1,1)$-modal $\langle
(M_1,t), (M_2,u)\rangle_k$-asimulation} iff for any $i,j \in \{ 1,
2 \}$, any $\bar{a}_m, a,c \in U_i$, $\bar{b}_m, b, d \in U_j$ any
unary predicate letter $P \in \Theta$, the conditions
\eqref{E:c1}, \eqref{E:c2}, \eqref{E:c3}, \eqref{E:c4},
\eqref{E:c6} are satisfied together with the following condition:
\begin{align}
&((\bar{a}_m, a)\mathrel{A}(\bar{b}_m, b) \wedge b\mathrel{R_j}d \wedge m + 1 < k) \Rightarrow\notag\\
&\qquad\qquad\Rightarrow \exists c \in U_i(a\mathrel{R_{\Box i}}c
\wedge (\bar{a}_m, a, c)\mathrel{A}(\bar{b}_m, b,
d))\label{E:cc5}\tag{\text{p-box-1}}
\end{align}
}
\end{definition}
\begin{definition}\label{D:asim11}
{\em Let $(M_1, t)$, $(M_2, u)$ be two pointed $\Theta$-models. A
binary relation $A$ is called \emph{$(1,1)$-modal $\langle
(M_1,t), (M_2,u)\rangle$-asimulation} iff for any $i,j \in \{ 1, 2
\}$, any $a,c \in U_i$, $b, d \in U_j$, any unary predicate letter
$P \in \Theta$ the conditions \eqref{E:c11}, \eqref{E:c22},
\eqref{E:c33}, \eqref{E:c44}, \eqref{E:c66} are satisfied together
with the following condition:
\begin{align}
&(a\mathrel{A}b \wedge b\mathrel{R_j}d) \Rightarrow \exists c \in
U_i(a\mathrel{R_{\Box i}}c \wedge
c\mathrel{A}d)\label{E:cc55}\tag{\text{box-1}}
\end{align}
}
\end{definition}
If, instead of using clause \eqref{E:diam1}, one chooses clause
\eqref{E:diam2}, things get somewhat more complicated. In order
get the right extensions of the basic asimulation notion, one has
to re-define asimulations as relation pairs rather than single
binary relations. As our first case we consider the set of
correspondence formulas induced by $(2,2)$-standard translations
of modal intuitionistic formulas:
\begin{definition}\label{D:k-asim22}
{\em Let $(M_1, t)$, $(M_2, u)$ be two pointed $\Theta$-models. An
ordered couple of binary relations $(A,B)$ is called
\emph{$(2,2)$-modal $\langle (M_1,t),
(M_2,u)\rangle_k$-asimulation} iff for any $i,j \in \{ 1, 2 \}$,
any $\bar{a}_m, a,c \in U_i$, $\bar{b}_m, b, d \in U_j$ any unary
predicate letter $P \in \Theta$, the conditions \eqref{E:c1},
\eqref{E:c2}, \eqref{E:c3}, \eqref{E:c4}, \eqref{E:c5} are
satisfied together with the following conditions:
\begin{align}
&B \subseteq \bigcup_{n > 0}((U_1^n \times U_2^n) \cup (U_2^n
\times U_1^n))\label{E:cc1}\tag{\text{p-B-type}}\\
&((\bar{a}_m, a)\mathrel{A}(\bar{b}_m, b) \wedge b\mathrel{R_j}d \wedge m + 1 < k) \Rightarrow\notag\\
&\qquad\qquad\Rightarrow \exists c \in U_i(a\mathrel{R_i}c \wedge
(\bar{a}_m, a, c)\mathrel{B}(\bar{b}_m, b, d))\label{E:cc6-1}\tag{\text{p-diam-2(1)}}\\
&((\bar{a}_m, a)\mathrel{B}(\bar{b}_m, b) \wedge a\mathrel{R_{\Diamond i}}c \wedge m < k) \Rightarrow\notag\\
&\qquad\qquad\Rightarrow \exists d \in U_j(b\mathrel{R_{\Diamond
j}}d \wedge (\bar{a}_m, a, c)\mathrel{A}(\bar{b}_m, b,
d))\label{E:cc6-2}\tag{\text{p-diam-2(2)}}
\end{align}
}
\end{definition}
\begin{definition}\label{D:asim22}
{\em Let $(M_1, t)$, $(M_2, u)$ be two pointed $\Theta$-models. An
ordered couple of binary relations $(A,B)$ is called
\emph{$(2,2)$-modal $\langle (M_1,t), (M_2,u)\rangle$-asimulation}
iff for any $i,j \in \{ 1, 2 \}$, any $a,c \in U_i$, $b, d \in
U_j$, any unary predicate letter $P \in \Theta$ the conditions
\eqref{E:c11}, \eqref{E:c22}, \eqref{E:c33}, \eqref{E:c44},
\eqref{E:c55} are satisfied together with the following
conditions:
\begin{align}
&B \subseteq (U_1 \times U_2) \cup (U_2\times
U_1)\label{E:cc11}\tag{\text{B-type}}\\
&(a\mathrel{A}b \wedge b\mathrel{R_j}d) \Rightarrow \exists c \in
U_i(a\mathrel{R_i}c \wedge
c\mathrel{B}d)\label{E:cc66-1}\tag{\text{diam-2(1)}}\\
&(a\mathrel{B}b \wedge a\mathrel{R_{\Diamond i}}c) \Rightarrow
\exists d \in U_j(b\mathrel{R_{\Diamond j}}d \wedge
c\mathrel{A}d)\label{E:cc66-2}\tag{\text{diam-2(2)}}
\end{align}
}
\end{definition}

The only case left is the one where one defines the satisfaction
relation by using \eqref{E:diam2} combined with \eqref{E:box1}.
The respective definitions simply re-shuffle the conditions
mentioned in the previous versions:

\begin{definition}\label{D:k-asim12}
{\em Let $(M_1, t)$, $(M_2, u)$ be two pointed $\Theta$-models. An
ordered couple of binary relation $(A,B)$ is called
\emph{$(1,2)$-modal $\langle (M_1,t),
(M_2,u)\rangle_k$-asimulation} iff for any $i,j \in \{ 1, 2 \}$,
any $\bar{a}_m, a \in U_i$, $\bar{b}_m, b, d \in U_j$ any unary
predicate letter $P \in \Theta$, the conditions \eqref{E:c1},
\eqref{E:cc1}, \eqref{E:c2}, \eqref{E:c3}, \eqref{E:c4},
\eqref{E:cc5}, \eqref{E:cc6-1}, and \eqref{E:cc6-2} are satisfied.
}
\end{definition}
\begin{definition}\label{D:asim12}
{\em Let $(M_1, t)$, $(M_2, u)$ be two pointed $\Theta$-models. An
ordered couple of binary relation $(A,B)$ is called
\emph{$(2,2)$-modal $\langle (M_1,t), (M_2,u)\rangle$-asimulation}
iff for any $i,j \in \{ 1, 2 \}$, any $a \in U_i$, $b, d \in U_j$,
any unary predicate letter $P \in \Theta$ the conditions
\eqref{E:c11}, \eqref{E:cc11}, \eqref{E:c22}, \eqref{E:c33},
\eqref{E:c44}, \eqref{E:cc55}, \eqref{E:cc66-1} and
\eqref{E:cc66-2} are satisfied.
}
\end{definition}

Before we reach the formulation of our main results, we still need
one more technical notion, that of invariance with respect to a
class of relations:

\begin{definition}\label{D:invariance}
{\em Let $\alpha$ be a class of relations such that for any $A \in
\alpha$ there is a $\Theta$ and there are $\Theta$-models $M_1$
and $M_2$ such that \eqref{E:c1} holds. Then a formula
$\varphi(x)$ is said to be \emph{invariant with respect to
$\alpha$}, iff for any $A \in \alpha$ for any corresponding
$\Theta$-models $M_1$ and $M_2$, and for any $a \in U_1$ and $b
\in U_2$ it is true that:
\[
(a\mathrel{A}b \wedge a \models_1 \varphi(x)) \Rightarrow b
\models_2 \varphi(x).
\]
}
\end{definition}
Of course, our primary examples of $\alpha$ will be the classes of
all $(i,j)$-modal $k$-asimulations for a given natural $k$ and the
classes of all $(i,j)$-modal asimulations. In case $j = 2$ we
identify the invariance with respect to $(A,B)$ with invariance
with respect to its left projection $A$.

It turns out that the properties of invariance with respect to
these relation classes can be used to characterize modal
intuitionistic fragment of FOL described by their respective
$ST_{ij}$. More precisely, one can obtain the following theorems:

\begin{theorem}\label{L:param}
Let $i, j \in \{ 1,2 \}$. A formula $\varphi(x)$ is equivalent to
an $(i,j)$-standard $x$-translation of a modal intuitionistic
formula iff there exists a $k \in \mathbb{N}$ such that
$\varphi(x)$ is invariant with respect to $(i,j)$-modal
$k$-asimulations.
\end{theorem}

\begin{theorem}\label{L:main}
Let $i, j \in \{ 1,2 \}$. A formula $\varphi(x)$ is equivalent to
an $(i,j)$-standard $x$-translation of an intuitionistic formula
iff $\varphi(x)$ is invariant with respect to $(i,j)$-modal
asimulations.
\end{theorem}

In what follows, we will also need to speak of \emph{basic
asimulations}. Using the notation of Definitions \ref{D:k-asim21}
and \ref{D:asim21}, a basic $\langle (M_1,t),
(M_2,u)\rangle_k$-asimulation (resp. a basic $\langle (M_1,t),
(M_2,u)\rangle$-asimulation) is a binary relation satisfying all
the conditions of Definition \ref{D:k-asim21} (resp. Definition
\ref{D:asim21}) but the last two.

\section{Characterization of modal intuitionistic formulas: the main case}\label{S:proof}

Theorems \ref{L:param} and \ref{L:main} admit of four different
instantiations of of $i$ and $j$. The most difficult ones seem to
be the two instantiations with $j = 2$ since with them the
situation bears the least degree of analogy to the first-order
intuitionistic logic. Therefore, in the present section consider
in some detail the case $i = j = 2$, whereas in the next one we
show how to adapt our proofs to the other cases.

In formulating our lemmas and presenting our proofs we will mimic
the structure of the proofs given in \cite{Ol13} and \cite{Ol14}
and will sometimes refer the reader to their respective parts in
cases concerning basic asimulations and their properties.

\subsection{Proof of Theorem \ref{L:param}}

\begin{lemma}[Key Lemma $1$]\label{L:asim}
Let $\varphi(x) = ST_{22}(I, x)$ for some modal intuitionistic
formula $I$, and let $r(\varphi) = k$. Let $\Sigma_\varphi
\subseteq \Theta$, let $(M_1, t)$, $(M_2, u)$ be two pointed
$\Theta$-models, and let $(A, B)$ be

\noindent a $(2,2)$-modal $\langle (M_1,t), (M_2,
u)\rangle_l$-asimulation. Then
\begin{align*}
((\bar{a}_m, a)\mathrel{A}(\bar{b}_m, b) \wedge m + k \leq l
\wedge a \models_i \varphi(x)) \Rightarrow b \models_j \varphi(x),
\end{align*}
for all $i,j \in \{ 1, 2 \}$, $(\bar{a}_m, a) \in U_i^{m + 1}$,
 and $(\bar{b}_m, b) \in U_j^{m + 1}$.
\end{lemma}
\begin{proof}
We proceed by induction on the complexity of $I$. In what follows
we will abbreviate the induction hypothesis by IH.

First we note, that since $(A, B)$ is a $(2,2)$-modal $\langle
(M_1,t), (M_2, u)\rangle_l$-asimulation, then $A$ is a basic
$\langle (M_1,t), (M_2, u)\rangle_l$-asimulation, which means that
we can re-use Lemma $1$ of \cite{Ol13} in order to handle the
basis and the induction steps for $\bot$, $\wedge$, $\vee$, and
$\to$. There remain the two cases which involve the modal
operators:

\emph{Case 1}. Let $I = \Box J$. Then
\[
\varphi(x) = \forall y(R(x, y) \to \forall z(R_\Box(y,z) \to
ST_{22}(J, z))).
\]

Assume that:

\begin{align}
&a \models_i \forall y(R(x, y) \to \forall z(R_\Box(y,z) \to
ST_{22}(J, z)))\label{E:1l0}\\
&(\bar{a}_m, a)\mathrel{A}(\bar{b}_m, b)\label{E:1l1}\\
&m + r(\varphi(x)) \leq l\label{E:1l2}
\end{align}
Moreover, it follows from definition of $r$ that:
\begin{align}
&r(\varphi(x)) \geq 2 \label{E:1l3}\\
&r(ST(J, z)) \leq r(\varphi(x)) - 2\label{E:1l4}
\end{align}
Now, consider arbitrary $d,f \in U_j$ such that $b\mathrel{R_j}d$
and $d\mathrel{R_{\Box j}}f$. Since \eqref{E:1l2} and
\eqref{E:1l3} clearly imply that $m + 1 < l$, it follows from
\eqref{E:1l1} and \eqref{E:c5} that one can choose $c,e \in U_i$,
such that:
\begin{align}
&a\mathrel{R_i}c\label{E:1l5}\\
&c\mathrel{R_{\Box i}}e\label{E:1l6}\\
&(\bar{a}_m, a,c,e)\mathrel{A}(\bar{b}_m, b,d,f)\label{E:1l7}
\end{align}
So, we reason as follows:
\begin{align}
&e \models_i ST_{22}(J, z)\label{E:1l8} &&\text{(from
\eqref{E:1l0}, \eqref{E:1l5}, and \eqref{E:1l6})}\\
&m + 2 + r(ST_{22}(J, z)) \leq l\label{E:1l9} &&\text{(from
\eqref{E:1l2} and \eqref{E:1l4})}\\
&f \models_j ST_{22}(J, z)\label{E:1l10} &&\text{(from
\eqref{E:1l7}, \eqref{E:1l8}, \eqref{E:1l9} by IH)}
\end{align}
Since $d$ was chosen to be an arbitrary $R_j$-successor of $b$,
and $f$ an arbitrary $R_{\Box j}$-successor of $d$, this means
that
\[
b \models_j \forall y(R(x, y) \to \forall z(R_\Box(y,z) \to
ST_{22}(J, z))),
\]
and we are done.

\emph{Case 2}. Let $I = \Diamond J$. Then

\[
\varphi(x) = \forall y(R(x, y) \to \exists z(R_\Diamond(y,z)
\wedge ST_{22}(J, z))).
\]

Assume that:

\begin{align}
&a \models_i \forall y(R(x, y) \to \exists z(R_\Diamond(y,z)
\wedge ST_{22}(J, z)))\label{E:2l0}\\
&(\bar{a}_m, a)\mathrel{A}(\bar{b}_m, b)\label{E:2l1}\\
&m + r(\varphi(x)) \leq l\label{E:2l2}
\end{align}
Moreover, it follows from definition of $r$ that:
\begin{align}
&r(\varphi(x)) \geq 2 \label{E:2l3}\\
&r(ST(J, y)) \leq r(\varphi(x)) - 2\label{E:2l4}
\end{align}
Since \eqref{E:2l2} and \eqref{E:2l3} clearly imply that $m + 1 <
l$, it follows from \eqref{E:2l1} and \eqref{E:cc6-1} that one can
choose a $c \in U_i$, such that:
\begin{align}
&a\mathrel{R_i}c\label{E:2l5}\\
&(\bar{a}_m, a, c)\mathrel{B}(\bar{b}_m, b, d)\label{E:2l6}
\end{align}
Now from \eqref{E:2l0} and \eqref{E:2l5} it follows that we can
choose an $e \in U_i$ such that
\begin{align}
&c\mathrel{R_{\Diamond i}}e\label{E:2l7}\\
&e \models_i ST_{22}(J, z)\label{E:2l8}
\end{align}
Also, by \eqref{E:2l6}, condition \eqref{E:cc6-2}, and the fact
that $m + 1 < l$, we have:
\begin{align}
&\exists f \in U_j(d\mathrel{R_{\Diamond j}}f \wedge (\bar{a}_m,
a, c, e)\mathrel{A}(\bar{b}_m, b, d, f))\label{E:2l9}
\end{align}
We further get that:
\begin{align}
&m + 2 + r(ST_{22}(J, z)) \leq l\label{E:2l10} &&\text{(from
\eqref{E:2l2} and \eqref{E:2l4})}\\
&\exists f \in U_j(d\mathrel{R_{\Diamond i}}f \wedge (f \models_j
ST_{22}(J, z)))\label{E:2l11}  &&\text{(by IH from \eqref{E:2l8},
\eqref{E:2l9}, and \eqref{E:2l10})}
\end{align}
Since $d$ was chosen to be an arbitrary $R_j$-successor of $b$,
this means that
\[
b \models_j \forall y(R(x, y) \to \exists z(R_\Diamond(y,z) \wedge
ST_{22}(J, z))),
\]
and we are done.
\end{proof}
We use Lemma \ref{L:asim} to derive the `easy' direction of
$(2,2)$-instantiation of Theorem \ref{L:param}:

\begin{corollary}\label{L:c-k-inv}
If $\varphi(x)$ is equivalent to a $(2,2)$-standard
$x$-translation of an intuitionistic formula, then there exists $k
\in \mathbb{N}$ such that $\varphi(x)$ is invariant with respect
to $(2,2)$-modal $k$-asimulations.
\end{corollary}
\begin{proof}
Let $\varphi(x)$ be logically equivalent to $ST_{22}(I,x)$ for
some intuitionistic formula $I$, and let $r(ST_{22}(I,x)) = k$.
Then it follows from Lemma \ref{L:asim} (setting $i: = 1$, $j :=
2$, $m := 0$, and $l := k$) that $ST_{22}(I,x)$ is invariant with
respect to $(2,2)$-modal $k$-asimulations, and so is $\varphi(x)$.
\end{proof}

On our way to the inverse direction of $(2,2)$-instantiation of
Theorem \ref{L:param} we first need a new piece of notation. For a
formula $\varphi(x)$ in the correspondence language, variable $x$
and a natural $l$, we denote with $int(\varphi, x, l)$ the set of
all $(2,2)$-standard $x$-translations of intuitionistic formulas,
which happen to be $(\Sigma_\varphi, x, l)$-formulas. We use this
notation to define three types of formulas which are important
components in the proofs to follow. Let $\Sigma_\varphi \subseteq
\Theta$, let $M$ be a $\Theta$-model and let $a \in U$. Then:
$$
tp_l(\varphi(x), M,a) = \{ \psi(x) \in int(\varphi,x,l)\mid M, a
\models \psi(x) \},
$$
$$
\overline{tp}_l(\varphi(x), M,a) = \{ \psi(x) \in
int(\varphi,x,l)\mid M,a \not\models \psi(x) \},
$$
and, further:
$$
imp_l(\varphi(x), M,g) = \bigcap\{ \overline{tp}_l(\varphi(x),M,b)
\mid \iota(R_{\Diamond})(a, b) \}
$$

We mention the following obvious link between different pieces of
notation just defined. Under the above-mentioned assumptions about
$\varphi(x)$, $l$, $\Theta$, $M$, $a$, and for any $M$ be a
$\Theta$-model $M'$ and $a' \in U'$ we have:
$$
tp_l(\varphi(x), M,a) \subseteq tp_l(\varphi(x), M',a')
\Leftrightarrow \overline{tp}_l(\varphi(x), M',a') \subseteq
\overline{tp}_l(\varphi(x), M,a).
$$

We then invoke the following well known-fact about classical
first-order logic:

\begin{lemma}\label{L:fin}
For any finite predicate vocabulary $\Theta$, any variable $x$ and
any natural $k$ there are, up to logical equivalence, only
finitely many $(\Theta, x, k)$-formulas.
\end{lemma}

This fact is proved as Lemma 3.4 in \cite[pp. 189--190]{EFT84}. It
implies that for every set of formulas, which have one of the
forms $int(\varphi, x, l)$, $tp_l(\varphi(x), M,a)$,
$\overline{tp}_l(\varphi(x), M,a)$, $imp_l(\varphi(x), M,a)$ there
exist a finite subset collecting the logical equivalents for all
the formulas in the set. Moreover, it allows us to collect logical
equivalents of all $(2,2)$-standard translations of intuitionistic
formulas which are true together with a given formula at some
pointed model in a single formula which we will call a complete
conjunction:

\begin{definition}\label{D:conj}
{\em Let $\varphi(x)$ be a formula. A conjunction $\Psi(x)$ of
formulas from $int(\varphi, x, k)$  is called a \emph{complete
$(\varphi, x, k)$-conjunction} iff there is a pointed model $(M,
a)$ such that $M, a \models \Psi(x) \wedge \varphi(x)$, and for
any $\psi(x) \in tp_k(\varphi(x),M,a)$ we have $\Psi(x)
\models\psi(x)$.}
\end{definition}

The two following lemmas summarize some rather obvious properties
of complete conjunctions:

\begin{lemma}\label{L:conj-ex}
For any formula $\varphi(x)$, any natural $k \geq 1$, any $\Theta$
such that $\Sigma_\varphi \subseteq \Theta$ and any pointed
$\Theta$-model $(M, a)$ such that $M, a \models \varphi(x)$ there
is a complete $(\varphi,x, k)$-conjunction $\Psi(x)$ such that $M,
a \models \Psi(x) \wedge \varphi(x)$.
\end{lemma}
\begin{proof}
Consider $tp_k(\varphi(x),M,a)$. This set is non-empty since
$ST_{22}(\bot \to \bot, x)$ will be true at $(M, a)$. Due to Lemma
\ref{L:fin}, we can choose in this set a non-empty finite subset
$\Gamma(x)$ such that any formula from $tp_k(\varphi(x),M,a)$ is
logically equivalent to (and hence follows from) a formula in
$\Gamma(x)$. By $\Gamma(x) \subseteq tp_k(\varphi(x),M,a)$, we
also have $M, a \models \bigwedge\Gamma(x)$, therefore,
$\bigwedge\Gamma(x)$ is a complete $(\varphi, x, k)$-conjunction.
\end{proof}

\begin{lemma}\label{L:conj-fin}
For any formula $\varphi(x)$ and any natural $k$ there are, up to
logical equivalence, only finitely many complete $(\varphi, x,
k)$-conjunctions.
\end{lemma}
\begin{proof}
It suffices to observe that for any formula $\varphi(x)$ and any
natural $k$, a complete $(\varphi, x, k)$-conjunction is a
$(\Sigma_\varphi, x, k)$-formula. Our lemma then follows from
Lemma \ref{L:fin}.
\end{proof}

As a result, we are now able to establish the `hard' right-to-left
direction of Theorem \ref{L:param}:
\begin{lemma}[Key Lemma $2$]\label{L:t1}
Let $k = r(\varphi(x))$ and let $\varphi(x)$ be invariant with
respect to $(2,2)$-modal $k$-asimulations. Then $\varphi(x)$ is
equivalent to a $(2,2)$-standard $x$-translation of a modal
intuitionistic formula.
\end{lemma}
\begin{proof}
We may assume that both $\varphi(x)$ and $\neg\varphi(x)$ are
satisfiable, since both $\bot$ and $\top$ are obviously invariant
with respect to $(2,2)$-modal $k$-asimulations and we have, for
example, the following valid formulas:

\begin{align*}
&\bot \leftrightarrow ST_{22}(\bot,x), \top \leftrightarrow
ST_{22}(\bot \to \bot, x).
\end{align*}

We  may also assume that there are two complete $(\varphi, x, k +
2)$-conjunctions $\Psi(x), \Psi'(x)$ such that
$\Psi'(x)\models\Psi(x)$, and both formulas $\Psi(x) \wedge
\varphi(x)$ and $\Psi'(x) \wedge \neg\varphi(x)$ are satisfiable.

For suppose otherwise. Then take the set of all complete
$(\varphi, x, k + 2)$-conjunctions $\Psi(x)$ such that the formula
$\Psi(x) \wedge \varphi(x)$ is satisfiable. This set is non-empty,
because $\varphi(x)$ is satisfiable, and by Lemma \ref{L:conj-ex},
it can be satisfied only together with some complete $(\varphi, x,
k + 2)$-conjunction. Now, using Lemma \ref{L:conj-fin}, choose in
it a finite non-empty subset $\{\,\Psi_{i_1}(x)\ldots,
\Psi_{i_n}(x)\,\}$ such that any complete $(\varphi, x, k +
2)$-conjunction is equivalent to an element of this subset. We can
show that $\varphi(x)$ is logically equivalent to
$\Psi_{i_1}(x)\vee\ldots \vee\Psi_{i_n}(x)$. In fact, if $M, a
\models \varphi(x)$ then, by Lemma \ref{L:conj-ex}, at least one
complete $(\varphi, x, k + 2)$-conjunction is true at $(M, a)$ and
therefore, its equivalent in $\{\,\Psi_{i_1}(x)\ldots,
\Psi_{i_n}(x)\,\}$ is also true at $(M, a)$, and so, finally we
have
 \[
 M, a \models \Psi_{i_1}(x)\vee\ldots \vee\Psi_{i_n}(x).
 \]
  In
the other direction, if $M, a \models \Psi_{i_1}(x)\vee\ldots
\vee\Psi_{i_n}(x)$, then for some $1 \leq j \leq n$ we have $M, a
\models \Psi_{i_j}(x)$. Then, since
$\Psi_{i_j}(x)\models\Psi_{i_j}(x)$ and by the choice of
$\Psi_{i_j}(x)$ the formula $\Psi_{i_j}(x) \wedge \varphi(x)$ is
satisfiable, so, by our assumption, the formula $\Psi_{i_j}(x)
\wedge \neg\varphi(x)$ must be unsatisfiable, and hence
$\varphi(x)$ must follow from $\Psi_{i_j}(x)$. But in this case we
will have $M, a \models \varphi(x)$ as well. So $\varphi(x)$ is
logically equivalent to $\Psi_{i_1}(x)\vee\ldots,
\vee\Psi_{i_n}(x)$ but the latter formula, being a disjunction of
conjunctions of $(2,2)$-standard $x$-translations of modal
intuitionistic formulas, is itself a $(2,2)$-standard
$x$-translation of a modal intuitionistic formula, and so we are
done.

If, on the other hand, one can take two complete $(\varphi, x, k +
2)$-conjunctions $\Psi(x), \Psi'(x)$ such that
$\Psi'(x)\models\Psi(x)$, and formulas $\Psi(x) \wedge \varphi(x)$
and $\Psi'(x) \wedge \neg\varphi(x)$ are satisfiable, we reason as
follows. Take a pointed $\Sigma_\varphi$-model $(M_1,t)$ such that
$t \models_1 \Psi(x) \wedge \varphi(x)$, and that any formula
$\psi(x) \in tp_{k + 2}(\varphi(x), M_1,t)$ follows from
$\Psi(x)$, and take any pointed model $(M_2, u)$ such that $u
\models_2 \Psi'(x) \wedge \neg\varphi(x)$.

We can construct an $(2,2)$-modal $\langle (M_1,t), (M_2,
u)\rangle_k$-asimulation and thus obtain a contradiction in the
following way. We define it as the ordered couple $(A, B)$, where
for arbitrary $i,j \in \{\,1, 2\,\}$ and  $(\bar{a}_m, a) \in
U^{m+1}_i$, $(\bar{b}_m,b) \in U^{m+1}_j$ we set:
$$
(\bar{a}_m,a)\mathrel{A}(\bar{b}_m,b) \Leftrightarrow (m \leq k
\wedge tp_{k - m + 2}(\varphi(x), M_i, a) \subseteq tp_{k - m +
2}(\varphi(x), M_j, b)),
$$
and, for $B$:
$$
(\bar{a}_m,a)\mathrel{B}(\bar{b}_m,b) \Leftrightarrow (m \leq k
\wedge imp_{k - m + 1}(\varphi(x), M_j, b) \subseteq imp_{k - m +
1}(\varphi(x),M_i, a)).
$$

With these definitions, we can show that $A$ is a basic $\langle
(M_1,t), (M_2, u)\rangle_k$-asimulation, arguing as in Theorem $1$
of \cite{Ol13}. Therefore, in order to show that $(A,B)$ is a
$(2,2)$-modal $\langle (M_1,t), (M_2, u)\rangle_k$-asimulation, we
only need to verify conditions \eqref{E:cc1}, \eqref{E:c5},
\eqref{E:cc6-1}, and \eqref{E:cc6-2}. It is clear that $B$
satisfies \eqref{E:cc1}.

To verify condition \eqref{E:c5}, take any
$(\bar{a}_m,a)\mathrel{A}(\bar{b}_m,b)$ such that $m + 1 < k$ and
any $d, f \in U_j$ such that $b\mathrel{R_j}d$ and
$d\mathrel{R_{\Box j}}f$. In this case we will also have $m + 2
\leq k$. Then consider $\overline{tp}_{k-m}(\varphi(x),M_j, f)$.
This set is non-empty, since by our assumption we have $k - m \geq
0$. Therefore, as we have $r(ST_{22}(\bot, x)) = 0$, we will also
have $ST_{22}(\bot, x) \in \overline{tp}_{k-m}(\varphi(x),M_j,
f)$. Then, according to our Lemma \ref{L:fin}, there exists a
finite non-empty set of logical equivalents for
$\overline{tp}_{k-m}(\varphi(x),M_j, f)$. Choosing this finite
set, we in fact choose some finite $\{\,ST_{22}(J_1,x)\ldots
ST_{22}(J_q, x)\,\} \subseteq \overline{tp}_{k-m}(\varphi(x),M_j,
f)$ such that
\begin{align*}
&\forall \psi(x) \in \overline{tp}_{k-m}(\varphi(x),M_j,
f)(\psi(x)\models ST_{22}(J_1,x)\vee\ldots \vee ST_{22}(J_q, x)).
\end{align*}
But then we obtain that
\[
b \not\models_j ST_{22}(\Box(J_1\vee\ldots \vee J_q), x).
\]
In fact, $d, f$ jointly falsify this boxed disjunction for $(M_j,
b)$. But, given that
\[
\{\,ST_{22}(J_1,x)\ldots ST_{22}(J_q, x)\,\} \subseteq
\overline{tp}_{k-m}(\varphi(x),M_j, f),
\]
the standard translation of boxed disjunction under consideration
must be in

\noindent $\overline{tp}_{k-m + 2}(\varphi,M_j, b)$. Note,
further, that by $(\bar{a}_m,a)\mathrel{A}(\bar{b}_m,b)$ we have
\[
tp_{k - m + 2}(\varphi(x),M_i, a) \subseteq tp_{k - m +
2}(\varphi(x),M_j, b),
\]
thus:
\[
\overline{tp}_{k - m + 2}(\varphi(x),M_j, b) \subseteq
\overline{tp}_{k - m + 2}(\varphi(x),M_i, a),
\]
 and therefore this boxed disjunction must be false at $(M_i,
a)$ as well. But then take any $c,e \in U_i$ such that
$a\mathrel{R_i}c$, $c\mathrel{R_{\Box i}}e$ and $c, e$ falsify the
boxed disjunction under consideration. By choice of
$\{\,ST(J_1,x)\ldots ST(J_q, x)\,\}$ it follows that
\[
\overline{tp}_{k - m}(\varphi(x),M_j, f) \subseteq
\overline{tp}_{k - m}(\varphi(x),M_i, e),
\]
and thus
\[
tp_{k - m}(\varphi(x),M_i, e) \subseteq tp_{k - m}(\varphi(x),M_j,
f),
\]
But then, again by the definition of $A$, and given the fact that
$m + 2\leq k$, we must also have $(\bar{a}_m,a, c,
e)\mathrel{A}(\bar{b}_m,b, d, f)$, and so condition \eqref{E:c5}
holds.

To verify condition \eqref{E:cc6-1}, take any
$(\bar{a}_m,a)\mathrel{A}(\bar{b}_m,b)$ such that $m + 1 < k$ and
any $d \in U_j$ such that $b\mathrel{R_j}d$. In this case we will
also have $m + 2 \leq k$. Then consider $imp_{k-m}(\varphi(x),M_j,
d)$. This set is non-empty, since by our assumption we have $k - m
\geq 0$. Therefore, as we have $r(ST_{22}(\bot, x)) = 0$, we will
also have $ST_{22}(\bot, x) \in imp_{k-m}(\varphi(x),M_j, d)$.
Then, according to our Lemma \ref{L:fin}, there exists a finite
non-empty set of logical equivalents for
$imp_{k-m}(\varphi(x),M_j, d)$. Choosing this finite set, we in
fact choose some finite
$$
\{\,ST_{22}(J_1,x)\ldots ST_{22}(J_q, x)\,\} \subseteq
imp_{k-m}(\varphi(x),M_j, d)
$$
such that
\begin{align*}
&\forall \psi(x) \in imp_{k-m}(\varphi(x),M_j, d)(\psi(x)\models
ST_{22}(J_1,x)\vee\ldots \vee ST_{22}(J_q, x)).
\end{align*}
But then we obtain that
\[
b \not\models_j ST_{22}(\Diamond(J_1\vee\ldots \vee J_q), x).
\]
In fact, $d$ falsify this disjunction for $(M_j, b)$. But, given
that
\[
\{\,ST_{22}(J_1,x)\ldots ST_{22}(J_q, x)\,\} \subseteq
imp_{k-m}(\varphi(x),M_j, d),
\]
the standard translation of the modalized disjunction under
consideration must be in

\noindent $\overline{tp}_{k-m + 2}(\varphi(x),M_j, b)$. Note,
further, that by $(\bar{a}_m,a)\mathrel{A}(\bar{b}_m,b)$ we have
\[
tp_{k - m + 2}(\varphi(x),M_i, a) \subseteq tp_{k - m +
2}(\varphi(x),M_j, b),
\]
thus:
\[
\overline{tp}_{k - m + 2}(\varphi(x),M_j, b) \subseteq
\overline{tp}_{k - m + 2}(\varphi(x),M_i, a),
\]
 and therefore the modalized disjunction must be false at $(M_i,
a)$ as well. But then take any $c \in U_i$ such that
$a\mathrel{R_i}c$, and for every $e$, such that
$c\mathrel{R_{\Diamond i}}e$, we have
$$
e \not\models_j ST_{22}(J_1,x)\vee\ldots \vee ST_{22}(J_q, x)
$$

By choice of $\{\,ST_{22}(J_1,x)\ldots ST_{22}(J_q, x)\,\}$ it
follows that
\[
imp_{k - m}(\varphi(x),M_j, d) \subseteq imp_{k -
m}(\varphi(x),M_i, c),
\]
But then, again by the definition of $B$, and given the fact that
$m + 2\leq k$, we must also have $(\bar{a}_m,a,
c)\mathrel{B}(\bar{b}_m,b, d)$, and so condition \eqref{E:cc6-1}
holds.

Finally, to verify condition \eqref{E:cc6-2}, take any
$(\bar{a}_m,a)\mathrel{A}(\bar{b}_m,b)$ such that $m < k$ and any
$c \in U_i$ such that $a\mathrel{R_{\Diamond i}}c$. In this case
we will also have $m + 1 \leq k$. Then consider $tp_{k-m +
1}(\varphi(x),M_i, c)$. This set is non-empty, since by our
assumption we have $k - m \geq 1$. Therefore, as we have
$r(ST_{22}(\bot \to \bot, x)) = 1$, we will also have
$ST_{22}(\bot  \to \bot, x) \in tp_{k-m + 1}(\varphi(x),M_i, c)$.
Then, according to our Lemma \ref{L:fin}, there exists a finite
non-empty set of logical equivalents for $tp_{k-m +
1}(\varphi(x),M_i, c)$. Choosing this finite set, we in fact
choose some finite $\{\,ST_{22}(I_1,x)\ldots ST_{22}(I_p, x)\,\}
\subseteq tp_{k-m + 1}(\varphi(x),M_i, c)$ such that
\begin{align*}
&\forall \psi(x) \in tp_{k-m + 1}(\varphi(x),M_i,
c)(ST_{22}(I_1,x)\wedge\ldots \wedge ST_{22}(I_p, x)\models
\psi(x)).
\end{align*}
But then we obtain that
\[
ST_{22}((I_1\wedge\ldots \wedge I_p), x) \notin imp_{k-m +
1}(\varphi(x),M_i, a).
\]
Note, further, that by $(\bar{a}_m,a)\mathrel{B}(\bar{b}_m,b)$ we
have
\[
imp_{k - m + 1}(\varphi(x),M_j, b) \subseteq imp_{k - m +
1}(\varphi(x),M_i, a),
\]
thus:
\[
ST_{22}((I_1\wedge\ldots \wedge I_p), x) \notin imp_{k-m +
2}(\varphi(x),M_j, b).
\]
 But then take any $d \in U_j$ such that
$b\mathrel{R_{\Diamond j}}d$ and we have
$$
d \models_j ST_{22}(I_1,x)\wedge\ldots \wedge ST_{22}(I_p, x).
$$

By choice of $\{\,ST_{22}(I_1,x)\ldots ST_{22}(I_p, x)\,\}$ it
follows that
\[
tp_{k - m + 1}(\varphi(x),M_i, c) \subseteq tp_{k - m +
1}(\varphi(x),M_j, d),
\]
But then, again by the definition of $A$, and given the fact that
$m + 1\leq k$, we must also have $(\bar{a}_m,a,
c)\mathrel{A}(\bar{b}_m,b, d)$, and so condition \eqref{E:cc6-2}
holds.

Therefore $(A, B)$ is a $(2,2)$-modal $\langle (M_1,t), (M_2,
u)\rangle_k$-asimulation, and we have got our contradiction in
place.
\end{proof}

We can now finish the proof of $(2,2)$-instantiation of Theorem
\ref{L:param}:

\begin{proof}
The  `only if' direction we have by Corollary \ref{L:c-k-inv}. In
the other direction, let $\varphi(x)$ be invariant with respect to
$k$-asimulations for some $k$. If $k \leq r(\varphi)$, then every
$(2,2)$-modal $r(\varphi)$-asimulation is a $(2,2)$-modal
$k$-asimulation, so $\varphi(x)$ is invariant with respect to
$(2,2)$-modal $r(\varphi)$-asimulations and hence, by Lemma
\ref{L:t1}, $\varphi(x)$ is equivalent to a standard
$x$-translation of an intuitionistic formula. If, on the other
hand, $r(\varphi) < k$, then set $l := k - r(\varphi)$ and
consider variables $\bar{y}_l$ not occurring in $\varphi(x)$. Then
$r(\forall\bar{y}_l\varphi(x)) = k$ and $\varphi(x)$ is logically
equivalent to $\forall\bar{y}_l\varphi(x)$, so the latter formula
is also invariant with respect to $k$-asimulations, and hence by
Theorem \ref{L:t1} $\forall\bar{y}_l\varphi(x)$ is logically
equivalent to a standard $x$-translation of an intuitionistic
formula. But then $\varphi(x)$ is equivalent to this standard
$x$-translation as well.
\end{proof}

\subsection{Proof of Theorem \ref{L:main}}

We now turn to the proof of $(2,2)$-instantiation of Theorem
\ref{L:main}. The  `only if' direction we again have by Corollary
\ref{L:c-k-inv}:

\begin{corollary}\label{L:c-inv}
If $\varphi(x)$ is equivalent to a $(2,2)$-standard
$x$-translation of an intuitionistic formula, then $\varphi(x)$ is
invariant with respect to $(2,2)$-modal asimulations.
\end{corollary}
\begin{proof}
Let $\varphi(x)$ be logically equivalent to $ST_{22}(I,x)$ for
some intuitionistic formula $I$. For an arbitrary $\Theta
\supseteq \Sigma_\varphi$, $\Theta$-models $M_1$ and $M_2$, and
arbitrary $t \in U_1$, $u\in U_2$ let $(A,B)$ be a $(2,2)$-modal
$\langle (M_1,t), (M_2, u)\rangle$-asimulation, so that we have
$t\mathrel{A}u$. Assume that
$$
t \models_1 \varphi(x).
$$
Then consider the ordered couple $(A',B')$ such that:
\begin{align*}
&A' = \{\,\langle(\bar{a}_m, a),(\bar{b}_m,
b;)\rangle\mid\\
&\qquad\qquad\mid \exists i, j(\{\,i,j \,\} = \{\,1,2\,\} \wedge
\bar{a}_m, a \in U_i \wedge \bar{b}_m,b \in U_j \wedge
a\mathrel{A}b)\,\};
\end{align*}

\begin{align*}
&B' = \{\,\langle(\bar{a}_m, a),(\bar{b}_m,
b;)\rangle\mid\\
&\qquad\qquad\mid \exists i, j(\{\,i,j \,\} = \{\,1,2\,\} \wedge
\bar{a}_m, a \in U_i \wedge \bar{b}_m,b \in U_j \wedge
a\mathrel{B}b)\,\};
\end{align*}
It is straightforward to verify that $(A',B')$ is a $(2,2)$-modal
$\langle (M_1,t), (M_2, u)\rangle_k$-asimulation for every $k \in
\mathbb{N}$. Moreover, we still have $t\mathrel{A'}u$. By
Corollary \ref{L:c-k-inv}, there is a natural $k$, such that
$\varphi(x)$ is invariant with respect to $(2,2)$-modal
$k$-asimulation, therefore we have
$$
u \models_2 \varphi(x).
$$
Since the $(2,2)$-modal $\langle (M_1,t), (M_2,
u)\rangle$-asimulation $(A,B)$ was chosen arbitrarily, this means
that $\varphi(x)$ is invariant with respect to $(2,2)$-modal
asimulations.
\end{proof}

To proceed, we need to introduce some further notions and results
from classical model theory. For a model $M$ and $\bar{a}_n \in U$
let $[M, \bar{a}_n]$ be the extension of $M$ with $\bar{a}_n$ as
new individual constants interpreted as themselves. It is easy to
see that there is a simple relation between the truth of a formula
at a sequence of elements of a $\Theta$-model and the truth of its
substitution instance in an extension of the above-mentioned kind;
namely, for any $\Theta$-model $M$, any $\Theta$-formula
$\varphi(\bar{y}_n,\bar{w}_m)$ and any $\bar{a}_n,\bar{b}_m \in U$
it is true that:

\[
[M, \bar{a}_n], \bar{b}_m \models \varphi(\bar{a}_n,\bar{w}_m)
\Leftrightarrow M, \bar{a}_n, \bar{b}_m \models
\varphi(\bar{y}_n,\bar{w}_m).
\]

We will call a theory of $M$ (and write $Th(M)$) the set of all
first-order sentences true at $M$. We will call an $n$-type of $M$
a set of formulas $\Gamma(\bar{w}_n)$ consistent with $Th(M)$.

\begin{definition}
Let $M$ be a $\Theta$-model. $M$ is \emph{$\omega$-saturated} iff
for all $k \in \mathbb{N}$ and for all $\bar{a}_n \in U$, every
$k$-type $\Gamma(\bar{w}_k)$ of $[M, \bar{a}_n]$ is satisfiable in
$[M, \bar{a}_n]$.
\end{definition}

Definition of $\omega$-saturation normally requires satisfiability
of $1$-types only. However, our modification is equivalent to the
more familiar version: see e.g. \cite[Lemma 4.31, p. 73]{Doets96}.

It is known that every model can be elementarily extended to an
$\omega$-saturated model; in other words, the following lemma
holds:

\begin{lemma}\label{L:ext}
Let $M$ be a $\Theta$-model. Then there is an $\omega$-saturated
extension $M'$ of $M$ such that for all $\bar{a}_n \in U$ and
every $\Theta$-formula $\varphi(\bar{w}_n)$:
\[
M, \bar{a}_n \models \varphi(\bar{w}_n) \Leftrightarrow M',
\bar{a}_n \models \varphi(\bar{w}_n).
\]
\end{lemma}
The latter lemma is a trivial corollary of e.g. \cite[Lemma
5.1.14, p. 216]{ChK73}.

A very useful property of $\omega$-saturated models is that one
can define among them $(2,2)$-modal asimulations more or less
according to the strategy assumed in the proof of Lemma
\ref{L:t1}. In order to do this, however, we need to re-define the
types used in the above proof. We collect the required changes in
the following definition:

\begin{definition}\label{D:inttheory}
Let $M$ be a $\Theta$-model, $t \in U$ and let $x$ be a variable
in correspondence language. Then we define $int_x(\Theta)$ to be
the set of all $\Theta$-formulas that are $(2,2)$-standard
$x$-translations of modal intuitionistic formulas. We further set:
$$
tp_x(M,t) = \{ \varphi(x) \in int_x(\Theta) \mid M,t \models
\varphi(x) \};
$$
$$
\overline{tp}_x(M,t) = \{ \varphi(x) \in int_x(\Theta) \mid M,t
\not\models \varphi(x) \};
$$
$$
imp_x(M,t) = \bigcap\{ \overline{tp}_x(M,u)
\mid\iota(R_\Diamond)(t,u) \}.
$$
\end{definition}

The analogue of `contrapositive' scheme mentioned above holds,
namely, for arbitrary models $M$, $M'$, and elements $a \in U$,
and $a' \in U'$ we have:
$$
tp_x(M,a) \subseteq tp_x(M',a') \Leftrightarrow
\overline{tp}_x(M',a') \subseteq \overline{tp}_x(M,a).
$$

The following lemma gives the precise version of the above
statement:

\begin{lemma}[Key Lemma $3$]\label{L:sat}
Let $M_1$, $M_2$ be $\omega$-saturated $\Theta$-models, let $t \in
U_1$, and $t' \in U_2$ be such that $tp(M_1, t) \subseteq tp(M_2,
t')$. Let $x$ be a variable in correspondence language. Then the
ordered couple $(A,B)$ such that:
\[
A = \{\,\langle a,b\rangle\mid \exists i,j (\{\,i,j\,\} = \{\,1,
2\,\} \wedge tp_x(M_i,a) \subseteq tp_x(M_j,b))\,\}
\]
and:
\[
B = \{\,\langle a,b\rangle\mid \exists i,j (\{\,i,j\,\} = \{\,1,
2\,\} \wedge imp_x(M_i,a) \supseteq imp_x(M_j,b))\,\}
\]
 is a $(2,2)$-modal $\langle(M_1, t),(M_2, t')\rangle$-asimulation.
\end{lemma}

\begin{proof}
It is evident that $B$ satisfies \eqref{E:cc11}, and arguing as in
Lemma 8 of \cite{Ol14}, we can show that $A$ is a basic
$\langle(M_1, t),(M_2, t')\rangle$-asimulation. So we are left to
consider the last three conditions of Definition \ref{D:asim22}.

To verify \emph{condition \eqref{E:c55}}, choose any $i$, $j$ such
that $\{\,i,j\,\} = \{\,1, 2\,\}$, any $a \in U_i$, $b \in U_j$
such that $a\mathrel{A}b$, that is to say, $tp_x(M_i, a)\subseteq
tp_x(M_j,b)$ and choose any $d,f \in U_j$ for which we have
$b\mathrel{R_j}d$ and $d\mathrel{R_{\Box j}}f$.

Consider $\overline{tp}_x(M_j,f)$. If $\{\,ST_{22}(J_1,x)\ldots
ST_{22}(J_q, x)\,\}$ is a finite subset of this type, then we have
$$
b \not\models_j ST(\Box(J_1 \vee\ldots \vee J_q), x).
$$
Since by contraposition of $a\mathrel{A}b$ we have that
$\overline{tp}_x(M_j, b)\subseteq \overline{tp}_x(M_i,a)$, we
obtain that
$$
a \not\models_i ST(\Box(J_1 \vee\ldots \vee J_q), x).
$$
This means that every finite subset of the type

\[ \{\,R(a, y),
R_\Box(y,x)\,\} \cup \{ \neg\psi(x) \mid \psi(x) \in
\overline{tp}_x(M_j,f) \}
\]
 is satisfiable at $[M_i, a]$. Therefore, by compactness of first-order logic, this
set is consistent with $Th([M_i, a])$ and, by $\omega$-saturation
of both $M_1$ and $M_2$, it must be satisfied in $[M_i, a]$ by
some $c,e \in U_i$. So for any such $c$ and $e$ we will have
$a\mathrel{R_i}c$, $c\mathrel{R_{\Box i}}e$ and, moreover,
\[
\forall \psi \in \overline{tp}_x(M_j,f)(e \not\models_i \psi(x)).
\]
Thus we have that $\overline{tp}_x(M_j, f) \subseteq
\overline{tp}_x(M_i,e)$, and further, by contraposition, that
$tp_x(M_i,e) \subseteq tp_x(M_j,f)$. Thus we get that
$e\mathrel{A}f$ and condition \eqref{E:c55} is verified.

To verify \emph{condition \eqref{E:cc66-1}}, choose any $i$, $j$
such that $\{\,i,j\,\} = \{\,1, 2\,\}$, any $a \in U_i$, $b \in
U_j$ such that $a\mathrel{A}b$, that is to say, $tp_x(M_i,
a)\subseteq tp_x(M_j,b)$ and choose any $d \in U_j$ for which we
have $b\mathrel{R_j}d$.

Consider $imp_x(M_j,d)$. If $\{\,ST_{22}(J_1,x)\ldots ST_{22}(J_q,
x)\,\}$ is a finite subset of this type, then we have
$$
b \not\models_j ST(\Diamond(J_1 \vee\ldots \vee J_q), x).
$$
Since by contraposition of $a\mathrel{A}b$ we have that
$\overline{tp}_x(M_j, b)\subseteq \overline{tp}_x(M_i,a)$, we
obtain that
$$
a \not\models_i ST(\Diamond(J_1 \vee\ldots \vee J_q), x).
$$
This means that every finite subset of the type

\[ \{\,R(a, x)\,\} \cup \{ \forall y(R_\Diamond(x,y) \to\neg\psi(y))\mid \psi(x) \in imp_x(M_j,d) \}
\]
 is satisfiable at $[M_i, a]$. Therefore, by compactness of first-order logic, this
set is consistent with $Th([M_i, a])$ and, by $\omega$-saturation
of both $M_1$ and $M_2$, it must be satisfied in $[M_i, a]$ by
some $c \in U_i$. So for any such $c$ we will have
$a\mathrel{R_i}c$ and, moreover,
\[
(\forall \psi(x) \in imp_x(M_j,d))(\forall e \in U_i)(R_{\Diamond
i}(c,e) \Rightarrow e \not\models_i \psi(x)).
\]
Thus we have that $imp_x(M_j, d) \subseteq imp_x(M_i,c)$, and
therefore, that $c\mathrel{B}d$. Thus condition \eqref{E:cc66-1}
is verified.

Finally, to verify \emph{condition \eqref{E:cc66-2}}, choose any
$i$, $j$ such that $\{\,i,j\,\} = \{\,1, 2\,\}$, any $a \in U_i$,
$b \in U_j$ such that $a\mathrel{B}b$, that is to say, $imp_x(M_j,
b) \subseteq imp_x(M_i,a)$ and choose any $c \in U_i$ for which we
have $a\mathrel{R_{\Diamond i}}c$.

Consider $tp_x(M_i,c)$. If $\{\,ST_{22}(I_1,x)\ldots ST_{22}(I_p,
x)\,\}$ is a finite subset of this type, then we have
$$
c \models_i ST((I_1 \wedge\ldots \wedge I_p), x).
$$
Therefore, the set $\{\,ST_{22}(I_1,x)\ldots ST_{22}(I_p, x)\,\}$
is disjoint from $imp_x(M_i,a)$, and thus it is also disjoint from
$imp_x(M_j,b)$. Therefore, the formula $ST(\Diamond(I_1
\wedge\ldots \wedge I_p), x)$ is also verified by some
$R_{\Diamond j}$-successor of $b$. More formally, this means that
every finite subset of the type

\[ \{\,R_\Diamond(b, x)\,\} \cup \{ \psi(x)\mid \psi(x) \in tp_x(M_i,c) \}
\]
 is satisfiable at $[M_j, b]$. Therefore, by compactness of first-order logic, this
set is consistent with $Th([M_j, b])$ and, by $\omega$-saturation
of both $M_1$ and $M_2$, it must be satisfied in $[M_j, b]$ by
some $d \in U_j$. So for any such $d$ we will have
$b\mathrel{R_{\Diamond j}}d$ and, moreover,
\[
(\forall \psi(x) \in tp_x(M_i,c))(d \models_j \psi(x)).
\]
Thus we have that $tp_x(M_i, c) \subseteq tp_x(M_j,c)$, and
therefore, that $c\mathrel{A}d$. The condition \eqref{E:cc66-2} is
verified.
\end{proof}

We are prepared now to prove the hard part of
$(2,2)$-instantiation of Theorem \ref{L:main}:

\begin{lemma}\label{L:hard}
Let $\varphi(x)$ be invariant with respect to $(2,2)$-modal
asimulations. Then $\varphi(x)$ is equivalent to $(2,2)$-standard
translation of a modal intuitionistic formula.
\end{lemma}
\begin{proof} We may assume that $\varphi(x)$ is
satisfiable, for $\bot$ is clearly invariant with respect to
$(2,2)$-modal asimulations and $\bot \leftrightarrow ST_{22}(\bot,
x)$ is a valid formula. Throughout this proof, will write
$ic(\varphi(x))$ for the following set:
$$
\{ \psi(x) \in int_x(\Sigma_\varphi) \mid \varphi(x) \models
\psi(x) \}
$$

Our strategy will be to show that $ic(\varphi(x)) \models
\varphi(x)$. Once this is done, we will apply compactness of
first-order logic and conclude that $\varphi(x)$ is equivalent to
a finite conjunction of standard $(2,2)$-modal $x$-translations of
intuitionistic formulas and hence to a standard $x$-translation of
the corresponding intuitionistic conjunction.

To show this, take any $\Sigma_\varphi$-model $M_1$ and  $a \in
U_1$ such that $a \models_1 ic(\varphi(x))$. Then, of course, we
also have $ic(\varphi(x)) \subseteq tp_x(M_1,a)$. Such a model
exists, because $\varphi(x)$ is satisfiable and $ic(\varphi(x))$
will be satisfied in any model satisfying $\varphi(x)$. Then we
can also choose a $\Sigma_\varphi$-model $M_2$ and $b \in U_2$
such that $b \models_2 \varphi(x)$ and $tp_x(M_2, b) \subseteq
tp_x(M_1, a)$.

For suppose otherwise. Then for any $\Sigma_\varphi$-model $M$
such that $U \subseteq \mathbb{N}$ and any $c \in U$ such that $M,
c \models \varphi(x)$ we can choose a modal intuitionistic formula
$I_{(M, c)}$ such that $ST_{22}(I_{(M, c)}, x)$ is in $tp_x(M, c)$
but not in $tp_x(M_1, a)$. Then consider the set
\[
S = \{\,\varphi(x)\,\} \cup \{\,\neg ST_{22}(I_{(M, c)}, x)\mid M,
c \models \varphi(x)\,\}
\]
Let $\{\,\varphi(x), \neg ST_{22}(I_1, x)\ldots , \neg
ST_{22}(I_q, x)\,\}$ be a finite subset of this set. If this set
is unsatisfiable, then we must have $\varphi(x) \models
ST_{22}(I_1, x)\vee\ldots \vee ST_{22}(I_q, x)$, but then we will
also have $(ST_{22}(I_1, x)\vee\ldots \vee ST_{22}(I_q, x)) \in
ic(\varphi(x)) \subseteq tp_x(M_1, a)$, and hence $(ST_{22}(I_1,
x)\vee\ldots \vee ST_{22}(I_q, x))$ will be true at $(M_1, a)$.
But then at least one of $ST_{22}(I_1, x)\ldots ST_{22}(I_q, x)$
must also be true at $(M_1, a)$, which contradicts the choice of
these formulas. Therefore, every finite subset of $S$ is
satisfiable, and, by compactness, $S$ itself is satisfiable as
well. But then, by the L\"{o}wenheim-Skolem property, we can take
a $\Sigma_\varphi$-model $M'$ such that $U' \subseteq \mathbb{N}$
and $g \in U'$ such that $S$ is true at $(M',g)$ and this will be
a model for which we will have both $M', g \models
ST_{22}(I_{(M',g)}, x)$ by choice of $I_{(M',g)}$ and $M',g
\not\models ST_{22}(I_{(M',g)}, x)$ by satisfaction of $S$, a
contradiction.

Therefore, we will assume in the following that some
$\Sigma_\varphi$-model $M_2$ and some $b \in U_2$ are such that $a
\models_1 ic(\varphi(x))$, $b \models_2 \varphi(x)$, and
$tp_x(M_2,b) \subseteq tp_x(M_1, a)$. According to Lemma
\ref{L:ext}, there exist $\omega$-saturated elementary extensions
$M'$, $M''$ of $M_1$ and $M_2$, respectively. We have:
\begin{align}
&M_1, a \models \varphi(x) \Leftrightarrow M', a \models
\varphi(x)\label{E:m1}\\
&M'', b \models \varphi(x)\label{E:m2}
\end{align}
Also, since $M_1$, $M_2$ are elementarily equivalent to $M'$,
$M''$, respectively, we have
\[
tp_x(M'',b) = tp_x(M_2,b) \subseteq tp_x(M_1, a) = tp_x(M', a).
\]
By $\omega$-saturation of $M'$, $M''$ and Lemma \ref{L:sat}, the
ordered couple $(A,B)$ such that:
\[
A = \{\,\langle c,d\rangle\mid \exists\mu,\mu'(\{\,\mu, \mu'\,\} =
\{\,M',M''\,\} \wedge tp_x(\mu,c)\subseteq tp_x(\mu',d))\,\}
\]
\[
B = \{\,\langle c,d\rangle\mid \exists\mu,\mu'(\{\,\mu, \mu'\,\} =
\{\,M',M''\,\} \wedge imp_x(\mu,c)\supseteq imp_x(\mu',d))\,\}
\]
is a $(2,2)$-modal $\langle(M'', b),(M', a)\rangle$-asimulation.
But then by \eqref{E:m2} and invariance of $\varphi(x)$ we get
$M', a \models \varphi(x)$, and further, by \eqref{E:m1} we
conclude that $M_1, a \models \varphi(x)$. Therefore, $\varphi(x)$
in fact follows from $ic(\varphi(x))$.
\end{proof}

Theorem \ref{L:main} now follows from Corollary \ref{L:c-inv} and
Lemma \ref{L:hard}.

\section{Other cases}\label{S:other}
We now briefly show how to obtain the proofs for the other three
instantiations of Theorems \ref{L:param} and \ref{L:main}. The
general scheme of the proofs in these cases is very similar to the
the proofs given in the previous section. The main difference is
that in the other cases we need to assume different sets of
conditions in the definitions of modal $k$-asimulations and modal
asimulations \emph{simpliciter}. This affects the three key lemmas
of the previous section, namely, Lemmas \ref{L:asim}, \ref{L:t1},
and \ref{L:sat}, in that some parts of their proofs become
irrelevant and some new parts need to be supplied instead.
Accordingly, when treating the other three instantiations of our
main results below, we mainly concentrate on how to revise the
proofs of these three lemmas.

\subsection{Case $i = 1$, $j = 2$}
In order to obtain the proofs of Theorems \ref{L:param} and
\ref{L:main} one needs to revise the proofs given in Section in
the following way:

\emph{Ad Key Lemma $1$}:

Revise the inductive step in case where $I = \Box J$ as follows:

In this case we have
\[
\varphi(x) = \forall y(R_\Box(x,y) \to ST_{12}(J, y))).
\]

Assume that:

\begin{align}
&a \models_i \forall y(R_\Box(x,y) \to ST_{12}(J, y)))\label{E:3l0}\\
&(\bar{a}_m, a)\mathrel{A}(\bar{b}_m, b)\label{E:3l1}\\
&m + r(\varphi(x)) \leq l\label{E:3l2}
\end{align}
Moreover, it follows from definition of $r$ that:
\begin{align}
&r(\varphi(x)) \geq 1 \label{E:3l3}\\
&r(ST(J, y)) \leq r(\varphi(x)) - 1\label{E:3l4}
\end{align}
Now, consider arbitrary $d \in U_j$ such that $b\mathrel{R_{\Box
j}}d$. Since \eqref{E:3l2} and \eqref{E:3l3} clearly imply that $m
< l$, it follows from \eqref{E:3l1} and \eqref{E:cc5} that one can
choose a $c \in U_i$, such that:
\begin{align}
&a\mathrel{R_{\Box i}}c\label{E:3l5}\\
&(\bar{a}_m, a,c)\mathrel{A}(\bar{b}_m, b,d)\label{E:3l6}
\end{align}
So, we reason as follows:
\begin{align}
&c \models_i ST_{12}(J, y)\label{E:3l7} &&\text{(from
\eqref{E:3l0} and \eqref{E:3l5})}\\
&m + 1 + r(ST_{12}(J, y)) \leq l\label{E:3l8} &&\text{(from
\eqref{E:3l2} and \eqref{E:3l4})}\\
&d \models_j ST_{12}(J, y)\label{E:3l9} &&\text{(from
\eqref{E:3l6}, \eqref{E:3l7}, \eqref{E:3l8} by IH)}
\end{align}
Since $d$ was chosen to be an arbitrary $R_{\Box j}$-successor of
$b$, this means that
\[
b \models_j \forall y(R_\Box(x,y) \to ST_{12}(J, y))),
\]
and we are done.

\emph{Ad Key Lemma $2$}:

Replace the verification of condition \eqref{E:c5} with the
following verification of condition \eqref{E:cc5}:

Take any $(\bar{a}_m,a)\mathrel{A}(\bar{b}_m,b)$ such that $m < k$
and any $d \in U_j$ such that $b\mathrel{R_{\Box j}}d$. In this
case we will also have $m + 1 \leq k$. Then consider
$\overline{tp}_{k-m}(\varphi(x),M_j, d)$. This set is non-empty,
since by our assumption we have $k - m \geq 0$. Therefore, as we
have $r(ST_{12}(\bot, x)) = 0$, we will also have $ST_{12}(\bot,
x) \in \overline{tp}_{k-m+ 1}(\varphi(x),M_j, d)$. Then, according
to Lemma \ref{L:fin}, there exists a finite non-empty set of
logical equivalents for $\overline{tp}_{k-m+ 1}(\varphi(x),M_j,
d)$. Choosing this finite set, we in fact choose some finite

\noindent$\{\,ST_{12}(J_1,x)\ldots ST_{12}(J_q, x)\,\} \subseteq
\overline{tp}_{k-m+ 1}(\varphi(x),M_j, d)$ such that
\begin{align*}
&\forall \psi(x) \in \overline{tp}_{k-m + 1}(\varphi(x),M_j,
d)(\psi(x)\models ST_{12}(J_1,x)\vee\ldots \vee ST_{12}(J_q, x)).
\end{align*}
But then we obtain that
\[
b \not\models_j ST_{12}(\Box(J_1\vee\ldots \vee J_q), x).
\]
In fact, $d$ falsifies this boxed disjunction for $(M_j, b)$. But,
given that
\[
\{\,ST_{12}(J_1,x)\ldots ST_{12}(J_q, x)\,\} \subseteq
\overline{tp}_{k-m+ 1}(\varphi(x),M_j, d),
\]
the standard translation of boxed disjunction under consideration
must be in

\noindent$\overline{tp}_{k-m + 2}(\varphi(x),M_j, b)$. Note,
further, that by $(\bar{a}_m,a)\mathrel{A}(\bar{b}_m,b)$ we have
\[
tp_{k - m + 2}(\varphi(x),M_i, a) \subseteq tp_{k - m +
2}(\varphi(x),M_j, b),
\]
thus:
\[
\overline{tp}_{k - m + 2}(\varphi(x),M_j, b) \subseteq
\overline{tp}_{k - m + 2}(\varphi(x),M_i, a),
\]
 and therefore this boxed disjunction must be false at $(M_i,
a)$ as well. But then take any $c \in U_i$ such that
$a\mathrel{R_{\Box i}}c$ and $c$ falsifies the disjunction under
consideration. By choice of $\{\,ST_{12}(J_1,x)\ldots ST_{12}(J_q,
x)\,\}$ it follows that
\[
\overline{tp}_{k - m + 1}(\varphi(x),M_j, d) \subseteq
\overline{tp}_{k - m + 1}(\varphi(x),M_i, c),
\]
and thus
\[
tp_{k - m + 1}(\varphi(x),M_i, c) \subseteq tp_{k - m +
1}(\varphi(x),M_j, d),
\]
But then, again by the definition of $A$, and given the fact that
$m + 1\leq k$, we must also have $(\bar{a}_m,a,
c)\mathrel{A}(\bar{b}_m,b, d)$, and so condition \eqref{E:cc5}
holds.

\emph{Ad Key Lemma $3$}:

Replace the verification of condition \eqref{E:c55} with the
following verification of condition \eqref{E:cc55}:

Choose any $i$, $j$ such that $\{\,i,j\,\} = \{\,1, 2\,\}$, any $a
\in U_i$, $b \in U_j$ such that $a\mathrel{A}b$, that is to say,
$tp_x(M_i, a)\subseteq tp_x(M_j,b)$ and choose any $d \in U_j$ for
which we have $b\mathrel{R_{\Box j}}d$.

Consider $\overline{tp}_x(M_j,d)$. If $\{\,ST_{12}(J_1,x)\ldots
ST_{12}(J_q, x)\,\}$ is a finite subset of this type, then we have
$$
b \not\models_j ST_{12}(\Box(J_1 \vee\ldots \vee J_q), x).
$$
Since by contraposition of $a\mathrel{A}b$ we have that
$\overline{tp}_x(M_j, b)\subseteq \overline{tp}_x(M_i,a)$, we
obtain that
$$
a \not\models_i ST_{12}(\Box(J_1 \vee\ldots \vee J_q), x).
$$
This means that every finite subset of the type

\[ \{\,R_\Box(a,x)\,\} \cup \{ \neg\psi(x) \mid \psi(x) \in
\overline{tp}_x(M_j,d) \}
\]
 is satisfiable at $[M_i, a]$. Therefore, by compactness of first-order logic, this
set is consistent with $Th([M_i, a])$ and, by $\omega$-saturation
of both $M_1$ and $M_2$, it must be satisfied in $[M_i, a]$ by
some $c \in U_i$. So for any such $c$ we will have
$a\mathrel{R_{\Box i}}c$ and, moreover,
\[
\forall \psi \in \overline{tp}_x(M_j,d)(c \not\models_i \psi(x)).
\]
Thus we have that $\overline{tp}_x(M_j, d) \subseteq
\overline{tp}_x(M_i,c)$, and further, by contraposition, that
$tp_x(M_i,c) \subseteq tp_x(M_j,d)$. Thus we get that
$c\mathrel{A}d$ and condition \eqref{E:c55} is verified.

\subsection{Case $i = 2$, $j = 1$}
The changes in three key lemmas in this case will look as follows:

\emph{Ad Key Lemma $1$}:

Revise the inductive step for the case $I = \Diamond J$ as
follows:

In this case we have
\[
\varphi(x) = \exists y(R_\Diamond(x,y) \wedge ST_{21}(J, y)).
\]

Assume that:

\begin{align}
&a \models_i \exists y(R_\Diamond(x,y) \wedge ST_{21}(J, y))\label{E:4l0}\\
&(\bar{a}_m, a)\mathrel{A}(\bar{b}_m, b)\label{E:4l1}\\
&m + r(\varphi(x)) \leq l\label{E:4l2}
\end{align}
Moreover, it follows from definition of $r$ that:
\begin{align}
&r(\varphi(x)) \geq 1 \label{E:4l3}\\
&r(ST(J, y)) \leq r(\varphi(x)) - 1\label{E:4l4}
\end{align}
Now, by \eqref{E:4l0} choose a $c \in U_i$ such that
\begin{align}
&a\mathrel{R_{\Diamond i}}c\label{E:4l5}\\
&c \models_i ST_{21}(J, y)\label{E:4l6}
\end{align}
Since \eqref{E:4l2} and \eqref{E:4l3} clearly imply that $m < l$,
it follows from \eqref{E:4l1} and \eqref{E:c6} that one can choose
a $d \in U_j$, such that:
\begin{align}
&b\mathrel{R_{\Diamond j}}d\label{E:4l7}\\
&(\bar{a}_m, a,c)\mathrel{A}(\bar{b}_m, b,d)\label{E:4l8}
\end{align}
So, we get that:
\begin{align}
&m + 1 + r(ST_{21}(J, y)) \leq l\label{E:4l9} &&\text{(from
\eqref{E:4l2} and \eqref{E:4l4})}\\
&d \models_j ST_{21}(J, y)\label{E:4l10} &&\text{(from
\eqref{E:4l6}, \eqref{E:4l8}, \eqref{E:4l9} by IH)}
\end{align}
Finally, from \eqref{E:4l7} and \eqref{E:4l10} infer that:
\[
b \models_j \exists y(R_\Diamond(x,y) \wedge ST_{21}(J, y)),
\]
and we are done.

\emph{Ad Key Lemma $2$}:

Replace the verification of conditions \eqref{E:cc6-1} and
\eqref{E:cc6-2} with the following verification of condition
\eqref{E:c6}:

Take any $(\bar{a}_m,a)\mathrel{A}(\bar{b}_m,b)$ such that $m < k$
and any $c \in U_i$ such that $a\mathrel{R_{\Diamond i}}c$. In
this case we will also have $m + 1 \leq k$. Then consider $tp_{k-m
+1}(\varphi(x),M_i, c)$. This set is non-empty, since by our
assumption we have $k - m + 1 \geq 1$. Therefore, as we have
$r(ST_{21}(\bot \to \bot, x)) = 1$, we will also have
$ST_{21}(\bot \to \bot, x) \in tp_{k-m+ 1}(\varphi(x),M_i, c)$.
Then, according to Lemma \ref{L:fin}, there exists a finite
non-empty set of logical equivalents for $tp_{k-m+
1}(\varphi(x),M_i, c)$. Choosing this finite set, we in fact
choose some finite $\{\,ST_{21}(I_1,x)\ldots ST_{21}(I_p, x)\,\}
\subseteq tp_{k-m+ 1}(\varphi(x),M_i, c)$ such that
\begin{align*}
&\forall \psi(x) \in tp_{k-m + 1}(\varphi(x),M_i,
c)(ST_{21}(I_1,x)\wedge\ldots \wedge ST_{21}(I_p, x)\models\psi(x)
).
\end{align*}
But then we obtain that
\[
a \models_i ST_{21}(\Diamond(I_1\wedge\ldots \wedge I_p), x).
\]
In fact, $c$ veriifies this modalized conjunction for $(M_i, a)$.
But, given that
\[
\{\,ST_{21}(I_1,x)\ldots ST_{21}(I_p, x)\,\} \subseteq tp_{k-m+
1}(\varphi(x),M_i, c),
\]
the standard translation of modalized conjunction under
consideration must be in $tp_{k-m + 2}(\varphi(x),M_i, a)$. Note,
further, that by $(\bar{a}_m,a)\mathrel{A}(\bar{b}_m,b)$ we have
\[
tp_{k - m + 2}(\varphi(x),M_i, a) \subseteq tp_{k - m +
2}(\varphi(x),M_j, b),
\]
 and therefore this modalized conjunction must be true at $(M_j,
b)$ as well. But then take any $d \in U_j$ such that
$b\mathrel{R_{\Diamond j}}d$ and $d$ verifies the conjunction
under consideration. By choice of $\{\,ST_{21}(I_1,x)\ldots
ST_{21}(I_p, x)\,\}$ it follows that
\[
tp_{k - m + 1}(\varphi(x),M_i, c) \subseteq tp_{k - m +
1}(\varphi(x),M_j, d),
\]
But then, again by the definition of $A$, and given the fact that
$m + 1\leq k$, we must also have $(\bar{a}_m,a,
c)\mathrel{A}(\bar{b}_m,b, d)$, and so condition \eqref{E:c6}
holds.

\emph{Ad Key Lemma $3$}:

Replace the verification of conditions \eqref{E:cc66-1} and
\eqref{E:cc66-2} with the following verification of condition
\eqref{E:c66}:

Choose any $i$, $j$ such that $\{\,i,j\,\} = \{\,1, 2\,\}$, any $a
\in U_i$, $b \in U_j$ such that $a\mathrel{A}b$, that is to say,
$tp_x(M_i, a)\subseteq tp_x(M_j,b)$ and choose any $c \in U_i$ for
which we have $a\mathrel{R_{\Diamond i}}c$.

Consider $tp_x(M_i,c)$. If $\{\,ST_{21}(I_1,x)\ldots ST_{21}(I_p,
x)\,\}$ is a finite subset of this type, then we have
$$
a \models_i ST_{21}(\Diamond(I_1 \wedge\ldots \wedge I_p), x).
$$
By $tp_x(M_i, a)\subseteq tp_x(M_j,b)$, we obtain that
$$
b \models_j ST_{12}(\Diamond(I_1 \wedge\ldots \wedge I_p), x).
$$
This means that every finite subset of the type

\[ \{\,R_\Diamond(b,x)\,\} \cup \{ \psi(x) \mid \psi(x) \in
tp_x(M_i,c) \}
\]
 is satisfiable at $[M_j, b]$. Therefore, by compactness of first-order logic, this
set is consistent with $Th([M_j, b])$ and, by $\omega$-saturation
of both $M_1$ and $M_2$, it must be satisfied in $[M_j, b]$ by
some $d \in U_j$. So for any such $d$ we will have
$b\mathrel{R_{\Diamond j}}d$ and, moreover,
\[
\forall \psi \in tp_x(M_i,c)(d \models_j \psi(x)).
\]
Thus we have that $tp_x(M_i,c) \subseteq tp_x(M_j,d)$. Thus we get
that $c\mathrel{A}d$ and condition \eqref{E:c66} is verified.

Another important revision of the proofs given in Section for the
case $i = 2$, $j = 1$ is omission of every reference to relation
$B$, since asimulations are now being defined as single relations
rather than ordered couples of relations.

Finally, in order to accommodate the proofs in Section to the case
$i = j = 1$ one just needs to combine the revisions given in the
present sections in a straightforward way.

\section{Characterization modulo first-order definable classes of models}\label{S:Rest}

Theorem \ref{L:main} establishes a criterion for the equivalence
of first-order formula to a standard translation of intuitionistic
formula on arbitrary first-order models. But one may have a
special interest in a proper subclass $\varkappa$ of the class of
first-order models viewing the models which are not in this
subclass as irrelevant, non-intended etc. In this case one may be
interested in the criterion for equivalence of a given first-order
formula to a standard translation of an intuitionistic predicate
formula \emph{over} this particular subclass. It turns out that if
some parts of this subclass are first-order axiomatizable then
only a slight modification of our general criterion is necessary
to solve this problem.

To tighten up on terminology, we introduce the following
definitions:
\begin{definition}\label{D:k}
Let $\varkappa$ be a class of models. Then:
\begin{enumerate}
\item $\varkappa(\Theta) = \{\,M \in \varkappa\mid M\text{ is a
$\Theta$-model}\,\}$; \item $\varkappa(\Theta)$ is first-order
axiomatizable iff there is a set $Ax$ of $\Theta$-sentences, such
that a $\Theta$-model $M$ is in $\varkappa$ iff $M \models Ax$;
\item A set $\Gamma$ of $\Theta$-formulas is
$\varkappa$-satisfiable iff $\Gamma$ is satisfied in some model of
$\varkappa$; \item A $\Theta$-formula $\varphi$
$\varkappa$-follows from $\Gamma$ $(\Gamma \models^\varkappa
\varphi)$ iff $\Gamma \cup \{\,\neg\varphi\,\}$ is
$\varkappa$-unsatisfiable; \item $\Theta$-formulas $\varphi$ and
$\psi$ are $\varkappa$-equivalent iff $\varphi \models^\varkappa
\psi$ and $\psi \models^\varkappa \varphi$.
\end{enumerate}
\end{definition}
It is clear that for any class $\varkappa$, such that $Ax$
first-order axiomatizes $\varkappa(\Theta)$, any set $\Gamma$ of
$\Theta$-formulas and any $\Theta$-formula $\varphi$, $\Gamma$ is
$\varkappa$-satisfiable iff $\Gamma \cup Ax$ is satisfiable, and
$\Gamma \models^\varkappa \varphi$ iff $\Gamma \cup Ax \models
\varphi$.

We say, further, that a formula $\varphi(x)$ is
$\varkappa$-invariant with respect to $(i,j)$-modal asimulations
(where $i,j \in \{ 1,2 \}$) iff it is invariant with respect to
the class of $(i,j)$-modal asimulations connecting models in
$\varkappa$.

Now for the criterion of $\varkappa$-equivalence to an
$(i,j)$-standard translation of modal intuitionistic formula:
\begin{theorem}\label{L:int-main}
Let $\varkappa$ be a class of first-order models such that
$\varkappa(\Theta)$ is first-order axiomatizable for all finite
$\Theta$, and let $\varphi(x)$ be $\varkappa$-invariant with
respect to $(i,j)$-modal asimulations. Then $\varphi(x)$ is
$\varkappa$-equivalent to an $(i,j)$-standard translation of a
modal intuitionistic formula.
\end{theorem}

\begin{proof} Let $Ax$ be the set of first-order
sentences that axiomatizes $\varkappa(\Sigma_\varphi)$. We may
assume that $\varphi(x)$ is
$\varkappa(\Sigma_\varphi)$-satisfiable, otherwise $\varphi(x)$ is
$\varkappa$-equivalent to $ST_{ij}(\bot, x)$ and we are done. In
what follows we will write $ic_\varkappa(\varphi(x))$ for the set
$$
\{ \psi(x) \in int_x(\Sigma_\varphi) \mid \varphi(x)
\models^\varkappa \psi(x) \}
$$

Our strategy will be to show that $ic_\varkappa(\varphi(x))
\models^\varkappa \varphi(x)$. Once this is done we will conclude
that
\[
Ax \cup ic_\varkappa(\varphi(x)) \models \varphi(x).
\]
Then we apply compactness of first-order logic and obtain that
$\varphi(x)$ is equivalent to a finite conjunction
$\bigwedge\Psi(x)$ of formulas from this set. But it follows then
that $\varphi(x)$ is $\varkappa$-equivalent to the conjunction of
the set $ic_\varkappa(\varphi(x)) \cap \Psi(x)$. In fact, by our
choice of $ic_\varkappa(\varphi(x))$ we have
\[
\varphi(x) \models^\varkappa \bigwedge(ic_\varkappa(\varphi(x))
\cap \Psi(x)),
\]
And by $\Psi(x) \subseteq Ax \cup ic_\varkappa(\varphi(x))$ we
have

\[
Ax \cup (ic_\varkappa(\varphi(x)) \cap \Psi(x)) \models \varphi(x)
\]
and hence
\[
ic_\varkappa(\varphi(x)) \cap \Psi(x) \models^\varkappa
\varphi(x).
\]

To show that $ic_\varkappa(\varphi(x)) \models^\varkappa
\varphi(x)$, take any $\Sigma_\varphi$-model $M_1$ and $a \in U_1$
such that $M_1 \in \varkappa$ and $a \models_1
ic_\varkappa(\varphi(x))$. Then, of course, we will also have
$ic_\varkappa(\varphi(x)) \subseteq tp_x(M_1,a)$. Such a model
exists, because $\varphi(x)$ is
$\varkappa(\Sigma_\varphi)$-satisfiable and
$ic_\varkappa(\varphi(x))$ will be $\varkappa$-satisfied in any
$\Sigma_\varphi$-model satisfying $\varphi(x)$. Then we can also
choose a $\Sigma_\varphi$-model $M_2$ and $b \in U_2$ such that
$M_2 \in \varkappa$, $b \models_2 \varphi(x)$, and $tp_x(M_2, b)
\subseteq tp_x(M_1, a)$.

For suppose otherwise. Then for any $\Sigma_\varphi$-model $M \in
\varkappa$ such that $U \subseteq \mathbb{N}$ and any $c \in U$
such that $M, c \models \varphi(x)$ we can choose a modal
intuitionistic formula $I_{(M, c)}$ such that $ST_{ij}(I_{(M, c)},
x)$ is in $tp_x(M, c)$ but not at $tp_x(M_1, a)$. Then consider
the set
\[
S = \{\,\varphi(x)\,\} \cup \{\,\neg ST_{ij}(I_{(M, c)}, x)\mid M
\in \varkappa, M, c \models \varphi(x)\,\}
\]
Let $\{\,\varphi(x), \neg ST_{ij}(I_1, x)\ldots , \neg
ST_{ij}(I_q, x)\,\}$ be a finite subset of this set. If this set
is $\varkappa$-unsatisfiable, then we must have
\[
\varphi(x) \models^\varkappa ST_{ij}(I_1, x)\vee\ldots \vee
ST_{ij}(I_q, x),
\]
but then we will also have
\[
(ST_{ij}(I_1, x)\vee\ldots \vee ST_{ij}(I_q, x)) \in
ic_\varkappa(\varphi(x)) \subseteq tp_x(M_1, a),
\]
and hence $(ST_{ij}(I_1, x)\vee\ldots \vee ST_{ij}(I_q, x))$ will
be true at $(M_1, a)$. But then at least one of $ST_{ij}(I_1,
x)\ldots ,ST_{ij}(I_q, x)$ must also be true at $(M_1, a)$, which
contradicts the choice of these formulas. Therefore, every finite
subset of $S$ is $\varkappa$-satisfiable. But then every finite
subset of the set $S \cup Ax$ is satisfiable as well. By
compactness of first-order logic $S \cup Ax$ is satisfiable, and,
by L\"{o}wenheim-Skolem property of first-order logic, there is a
$\Sigma_\varphi$-model $M'$ and $g \in U'$ such that $U' \subseteq
\mathbb{N}$ and $(M',g)$ satisfies $S \cup Ax$. But then we must
have $M' \in \varkappa$, since  $M'$ is a $\Sigma_\varphi$-model
satisfying the set of axioms for $\varkappa(\Sigma_\varphi)$.

For this model and for this element in it we will have both $M', g
\models ST_{ij}(I_{(M',g)}, x)$ by choice of $I_{(M',g)}$ and $M',
g \not\models ST_{ij}(I_{(M',g)}, x)$ by the satisfaction of $S$,
a contradiction.

Therefore, for any given $\Sigma_\varphi$-model $M_1$ such that
$M_1 \in \varkappa$ and for any $a\in U_1$ satisfying
$ic_\varkappa(\varphi(x))$  we can choose a $\Sigma_\varphi$-model
$M_2$ and $b \in U_2$ such that $M_2 \in \varkappa$, $b \models_2
\varphi(x)$, and $tp_x(M_2, b) \subseteq tp_x(M_1,a)$. Then,
reasoning exactly as in the proof of Theorem \ref{L:hard}, we
conclude that $a \models_1 \varphi(x)$. Therefore, $\varphi(x)$ in
fact $\varkappa$-follows from $ic_\varkappa(\varphi(x))$.
\end{proof}

\begin{theorem}\label{L:int-final}
Let $\varkappa$ be a class of first-order models such that for all
finite $\Theta$ class $\varkappa(\Theta)$ is first-order
axiomatizable. Then a formula $\varphi(x)$ is
$\varkappa$-invariant with respect to $(i,j)$-modal asimulations
iff it is $\varkappa$-equivalent to an $(i,j)$-standard
translation of a modal intuitionistic formula.
\end{theorem}
\begin{proof} From left to right our theorem follows from
Theorem \ref{L:int-main}. In the other direction, assume that
$\varphi(x)$ is $\varkappa$-equivalent to $ST_{ij}(I,x)$ and
assume that for some $\Theta$ such that $\Sigma_\varphi \subseteq
\Theta$, some $\Theta$-models $M_1$, $M_2$, and some $a \in U_1$
and $b \in U_2$ such that $M_1,M_2 \in \varkappa$, $A$ is a
$(i,j)$-modal $\langle(M_1, a),(M_2, b)\rangle$-asimulation and $a
\models_1 \varphi(x)$. Then, since $ST_{ij}(I,x)$ is
$\varkappa$-equivalent to $\varphi(x)$ and $M_1$ is in
$\varkappa$, we also have $a \models_1 ST_{ij}(I,x)$. By Theore
\ref{L:main} it follows that $b \models_2 ST_{ij}(I,x)$, but since
$ST(I,x)$ is $\varkappa$-equivalent to $\varphi(x)$ and $M_2$ is
in $\varkappa$, we also have $b \models_2\varphi(x)$. Therefore,
$\varphi(x)$ is $\varkappa$-invariant with respect to
asimulations.
\end{proof}

Theorems \ref{L:int-main} and \ref{L:int-final} imply that
$(i,j)$-modal asimulations as criteria for equivalence to a
standard translation of a modal intuitionistic formula are easily
scalable down to any first-order definable class of models. As all
the conditions on intended models for intuitionistic modal logic,
that were presented in the existing literature, seem to be easily
first-order definable, this means that the criterion for
equivalence of a formula in correspondence language to an
$(i,j)$-standard translation of modal intuitionistic formula over
the class of intended models will be just invariance with respect
to $(i,j)$-modal asimulations between the intended models.

\section{Conclusion}\label{S:conclusion}
In this paper we have defined and vindicated $4$ different
versions of modal asimulations, capturing the $4$ different
fragments of classical first-order logic induced by the
corresponding systems of satisfaction clauses for modal
intuitionistic logic. In doing so, we were concentrating on
different variants of Kripke semantics for intuitionistic
modalities, present in the existing literature.

However, it is easy to see, that our approach is but an instance
of a quite general algorithm that can be easily generalized to
deal with a much wider class of extensions of intuitionistic
propositional logic. It is not clear at the moment, how far such a
generalization is able to reach; we only mention here one
conjecture to give reader a flavor of what can be expected in the
way of future work along the research lines presented above.

It is easy to see that all the intuitionistic modalities
considered above are instances of a general scheme: the effect of
the modality on its only argument is always a superposition of
`guarded' quantifiers, that is to say, quantifiers, restricted by
a binary relation which is a part of Kripke model for this
modality. To be more precise, let us call a generalized
intuitionistic modality any modality $\mu I$ whose induced
standard translation looks as follows:
$$
ST(\mu I, x) = Q_ny_n(R_n(y_n,y_{n-1}\odot_n Q_{n-1}(\ldots
Q_1y_1(R_1(y_1,x) \odot_1 ST(I, x))\ldots)),
$$
where $n \geq 1$, and for all $i$ such that $1 \leq i \leq n$ we
assume $Q_i \in \{ \forall, \exists \}$. Further, $\odot_i = \to$
if $Q_i = \forall$ and $\odot_i = \wedge$ if $Q_i =
\exists$.\footnote{The relations $R_1,\ldots, R_n$ need not be all
different.}

Then our \textbf{\emph{conjecture}} is that to capture the
extension of intuitionistic propositional logic by modality $\mu$,
one generally needs to define asimulation as a tuple $(A_1,\ldots,
A_{k_\mu})$ of binary relations, where $k_\mu\leq n$ is the number
of quantifier alternations in $ST(\mu I, x)$. The set of
conditions on basic asimulation will then have to be enlarged by
$k_\mu$ new conditions $r_1,\ldots, r_{k_\mu}$, each having one of
the two following forms: either
$$
a_1\mathrel{A_p}b_1 \wedge \bigwedge^m_{s =
1}(\iota_j(S_s)(b_s,b_{s + 1})) \Rightarrow \exists a_2\ldots a_m
\in U_i(\bigwedge^m_{s = 1}(\iota_i(S_s)(a_s,a_{s + 1})) \wedge
a_m\mathrel{A_q}b_m),
$$
or
$$
a_1\mathrel{A_p}b_1 \wedge \bigwedge^m_{s =
1}(\iota_i(S_s)(a_s,a_{s + 1})) \Rightarrow \exists b_2\ldots b_m
\in U_j(\bigwedge^m_{s = 1}(\iota_j(S_s)(b_s,b_{s + 1})) \wedge
a_m\mathrel{A_q}b_m),
$$
for all $a_1,\ldots, a_m \in U_i$ and all $b_1,\ldots, b_m \in
U_j$. Here, $i$, $j$ have the same meaning as in Definition
\ref{D:asim21} and $\{ S_1,\ldots, S_m \} \subseteq \{ R_1,\ldots,
R_n \}$. Moreover, we will ensure that in the last condition
$r_{k_\mu}$ we will have $A_p = A_1$ and that the form of
$r_{k_\mu}$ will be the former of the above forms if $Q_n$ is
$\forall$ and the latter if $Q_n$ is $\exists$.

More precisely, we define $r_1,\ldots, r_{k_\mu}$ by induction on
$n$ as follows.

If $n = 1$ and $Q_1 = \forall$, we set:
$$
r_1 := (a_1\mathrel{A_1}b_1 \wedge \iota_j(R_1)(b_1,b_2))
\Rightarrow \exists a_2 \in U_i(\iota_i(R_1)(a_1,a_2) \wedge
a_2\mathrel{A_j}b_2)),
$$
and if $Q_1 = \exists$, we set
$$
r_1 := (a_1\mathrel{A_1}b_1 \wedge \iota_i(R_1)(a_1,a_2))
\Rightarrow \exists b_2 \in U_j(\iota_j(R_1)(b_1,b_2) \wedge
a_2\mathrel{A_j}b_2)).
$$
Incidentally, the fact that these singleton sets of conditions are
adequate already follows from the above proofs concerning
\eqref{E:box1} and \eqref{E:diam1}.

Assume, further, that $n = s + 1$ for $s \geq 1$.

As induction hypothesis, we suppose that the set of conditions
$r_1,\ldots, r_{k_{\mu^-}}$ for $\mu^-$, where

$$
ST(\mu^- I, x) = Q_{n-1}(\ldots Q_1y_1(R_1(y_1,x) \odot_1 ST(I,
x))\ldots))
$$
is already defined.

To define $r_1,\ldots, r_{k_\mu}$, we need to distinguish the
following cases:

\emph{Case 1}. $Q_n = Q_{n - 1}$. Then $k_{\mu} = k_{\mu^-}$ and
we do not need to introduce new conditions. Instead, we transform
condition $r_{k_{\mu^-}}$ in the following way:

\emph{Case 1.1} If $Q_n = Q_{n - 1} = \forall$, then by induction
hypothesis condition $r_{k_{\mu^-}}$  has the form:
$$
a_1\mathrel{A_1}b_1 \wedge \bigwedge^m_{s =
1}(\iota_j(S_s)(b_s,b_{s + 1})) \Rightarrow \exists a_2\ldots a_m
\in U_i(\bigwedge^m_{s = 1}(\iota_i(S_s)(a_s,a_{s + 1})) \wedge
a_m\mathrel{A_q}b_m).
$$
We then replace $r_{k_{\mu^-}}$ by the following rule
$r_{k_{\mu}}$, where we assume $a' \in U_i$, $b' \in U_j$:
\begin{align*}
&a'\mathrel{A_1}b' \wedge (\iota_j(R_n)(b',b_1) \wedge
\bigwedge^m_{s = 1}(\iota_j(S_s)(b_s,b_{s + 1}))
\Rightarrow\notag\\
&\qquad\qquad\qquad\Rightarrow \exists a_1\ldots a_m \in
U_i(\iota_i(R_n)(a',a_1) \wedge\bigwedge^m_{s =
1}(\iota_i(S_s)(a_s,a_{s + 1})) \wedge a_m\mathrel{A_q}b_m).
\end{align*}

\emph{Case 1.2} If $Q_n = Q_{n - 1} = \exists$, then by induction
hypothesis condition $r_{k_{\mu^-}}$  has the form:
$$
a_1\mathrel{A_1}b_1 \wedge \bigwedge^m_{s =
1}(\iota_i(S_s)(a_s,a_{s + 1})) \Rightarrow \exists b_2\ldots b_m
\in U_j(\bigwedge^m_{s = 1}(\iota_j(S_s)(b_s,b_{s + 1})) \wedge
a_m\mathrel{A_q}b_m).
$$
We then replace $r_{k_{\mu^-}}$ by the following rule
$r_{k_{\mu}}$, where we assume $a' \in U_i$, $b' \in U_j$:
\begin{align*}
&a'\mathrel{A_1}b' \wedge (\iota_i(R_n)(a',a_1) \wedge
\bigwedge^m_{s = 1}(\iota_i(S_s)(a_s,a_{s + 1}))
\Rightarrow\notag\\
&\qquad\qquad\qquad\Rightarrow \exists b_1\ldots b_m \in
U_j(\iota_j(R_n)(b',b_1) \wedge\bigwedge^m_{s =
1}(\iota_j(S_s)(a_s,a_{s + 1})) \wedge a_m\mathrel{A_q}b_m).
\end{align*}

\emph{Case 2}. $Q_n \neq Q_{n - 1}$. Then $k_{\mu} = k_{\mu^-} +
1$ and we need to increase the number of conditions by one. We do
this as follows:

\emph{Case 2.1} If $Q_n = \forall$, $Q_{n - 1} = \exists$, then by
induction hypothesis condition $r_{k_{\mu^-}}$  has the form:
$$
a_1\mathrel{A_1}b_1 \wedge \bigwedge^m_{s =
1}(\iota_i(S_s)(a_s,a_{s + 1})) \Rightarrow \exists b_2\ldots b_m
\in U_j(\bigwedge^m_{s = 1}(\iota_j(S_s)(b_s,b_{s + 1})) \wedge
a_m\mathrel{A_q}b_m).
$$

we then replace $r_{k_{\mu^-}}$ by the following condition
$r'_{k_{\mu^-}}$:
$$
a_1\mathrel{A_{k_\mu}}b_1 \wedge \bigwedge^m_{s =
1}(\iota_i(S_s)(a_s,a_{s + 1})) \Rightarrow \exists b_2\ldots b_m
\in U_j(\bigwedge^m_{s = 1}(\iota_j(S_s)(b_s,b_{s + 1})) \wedge
a_m\mathrel{A_q}b_m),
$$

and we add the following new condition $r_{k_\mu}$:
$$
a_1\mathrel{A_1}b_1 \wedge \iota_j(R_n)(b_1,b_2)) \Rightarrow
\exists a_2 \in U_i(\iota_i(R_n)(a_1,a_2) \wedge
a_2\mathrel{A_{k_\mu}}b_2).
$$

\emph{Case 2.2} If $Q_n = \exists$, $Q_{n - 1} = \forall$, then by
induction hypothesis condition $r_{k_{\mu^-}}$  has the form:
$$
a_1\mathrel{A_1}b_1 \wedge \bigwedge^m_{s =
1}(\iota_i(S_s)(a_s,a_{s + 1})) \Rightarrow \exists b_2\ldots b_m
\in U_j(\bigwedge^m_{s = 1}(\iota_j(S_s)(b_s,b_{s + 1})) \wedge
a_m\mathrel{A_q}b_m),
$$

we then replace $r_{k_{\mu^-}}$ by the following condition
$r'_{k_{\mu^-}}$:
$$
a_1\mathrel{A_{k_\mu}}b_1 \wedge \bigwedge^m_{s =
1}(\iota_i(S_s)(a_s,a_{s + 1})) \Rightarrow \exists b_2\ldots b_m
\in U_j(\bigwedge^m_{s = 1}(\iota_j(S_s)(b_s,b_{s + 1})) \wedge
a_m\mathrel{A_q}b_m),
$$

and we add the following new condition $r_{k_\mu}$:
$$
a_1\mathrel{A_1}b_1 \wedge \iota_j(R_n)(b_1,b_2)) \Rightarrow
\exists a_2 \in U_i(\iota_i(R_n)(a_1,a_2) \wedge
a_2\mathrel{A_{k_\mu}}b_2).
$$

It is straightforward to verify that our systems of rules for both
\eqref{E:box2} and \eqref{E:diam2} were generated according to the
Cases 1.1 and 2.1 of this inductive definition, respectively.
Note, however, that this general scheme is not always the most
effective one. For example, consider an extension of
intuitionistic propositional logic, where both \eqref{E:diam2} and
\eqref{E:diam1} are available. This would allow to considerably
simplify the corresponding notion of asimulation, indeed, one
could  get rid in this case of the second binary relation $B$ and
define asimulations as appropriate type of single binary relation
$A$.

The other thing worth noting is that the above mentioned scheme
for defining asimulations which capture expressive powers of
generalized intuitionistic modalities is neither the only nor the
most abstract among the generalizations that naturally come to
mind in this respect.

Therefore, in our future work, we hope both to provide
substantiation for the above-defined scheme concerning the
generalized intuitionistic modalities and further pursue the
manifold research opportunities opening along this research line.

\section{Acknowledgements}

To be inserted.

\bibliographystyle{habbrv}

\bibliography{intbib}

 }
\end{document}